\documentclass[a4paper,11pt]{article}
\usepackage{etex}
\usepackage{makeidx}
\usepackage[english]{babel}
\usepackage[latin1]{inputenc}
\usepackage{amsfonts}
\usepackage{bm}
\usepackage{amssymb}
\usepackage{latexsym}
\usepackage{amsmath}
\usepackage{amscd}
\usepackage{amsthm}
\usepackage{rotating}
\usepackage{graphicx}
\usepackage{pifont}
\usepackage{epsfig}
\usepackage{pstricks}
\usepackage{alltt}
\usepackage[all]{xy}
\usepackage{float}
\usepackage{epstopdf}
\usepackage{subfig}
\usepackage{wrapfig}
\usepackage{color}

\usepackage{lineno}
\usepackage{setspace}
\title{Nonparametric adaptive estimation of order 1 Sobol indices in stochastic models, with an application to Epidemiology\footnote{This study was funded by the French Agence Nationale de Recherche sur le Sida et les H\'epatites virales (ANRS),
grant number 95146 and by Labex CEMPI (ANR-11-LABX-0007-01). V.C.T. also acknowledge support from the Chaire ``Mod\'elisation Math\'ematique et
Biodiversit\'e'' of Veolia Environnement-Ecole Polytechnique-Museum National d'Histoire Naturelle-Fondation X and ANR CADENCE (ANR-16-CE32-0007).
The authors would like to thank the working group previously involved in the development of the model for HCV transmission among PWID:
Sylvie Deuffic-Burban, Jean-St\'ephane Dhersin, Marie Jauffret-Roustide and Yazdan Yazdanpanah.
Numerical results presented in this paper were carried out using the regional computational cluster supported by Universit\'e Lille 1,
CPER Nord-Pas-de-Calais/FEDER, France Grille, CNRS. We would like to thank the technical staff of the CRI-Lille 1 center.}}

\author{Gwena\"elle Castellan\footnote{Univ. Lille, CNRS, UMR 8524 - Laboratoire Paul Painlev\'e, F-59000 Lille, France; E-mail: \texttt{gwenaelle.castellan@math.univ-lille1.fr}}, Anthony Cousien\thanks{INSERM, IAME, UMR 1137, Universit\'e Paris Diderot, Sorbonne Paris Cit\'e, F-75018 Paris, France; E-mail: \texttt{anthony.cousien@gmail.com}}, Viet Chi Tran\thanks{LAMA, Univ Gustave Eiffel, UPEM, Univ Paris Est Creteil, CNRS, F-77447, Marne-la-Vall\'ee, France; E-mail: \texttt{chi.tran@u-pem.fr} \newline Authors are listed in alphabetical order.}}

\date{\today}

\numberwithin{equation}{section}

\def\E{\mathbb{E}}

\def\ind{{\mathchoice {\rm 1\mskip-4mu l} {\rm 1\mskip-4mu l}
{\rm 1\mskip-4.5mu l} {\rm 1\mskip-5mu l}}}

\newcommand{\be} {\begin{equation}}
\newcommand{\ee} {\end{equation}}
\newcommand{\bea} {\begin{eqnarray}}
\newcommand{\eea} {\end{eqnarray}}
\newcommand{\Bea} {\begin{eqnarray*}}
\newcommand{\Eea} {\end{eqnarray*}}

\theoremstyle{plain}
\newtheorem{theorem}{Theorem}[section]
\newtheorem{prop}[theorem]{Proposition}
\newtheorem{lemma}[theorem]{Lemma}
\newtheorem{corollary}[theorem]{Corollary}
\newtheorem{definition}[theorem]{Definition}

\begin{document}
\vspace{-1cm}
\maketitle
\vspace{-1cm}

\begin{abstract}
Global sensitivity analysis is a set of methods aiming at quantifying the
contribution of an uncertain input parameter of the model (or combination
of parameters) on the variability of the response. We consider here the
estimation of the Sobol indices of order 1 which are commonly-used
indicators based on a decomposition of the output's variance. In a
deterministic framework, when the same inputs always give the same outputs,
these indices are usually estimated by replicated simulations of the model.
In a stochastic framework, when the response given a set of input
parameters is not unique due to randomness in the model, metamodels are
often used to approximate the mean and dispersion of the response by
deterministic functions. We propose a new non-parametric estimator without
the need of defining a metamodel to estimate the Sobol indices of order 1.
The estimator is based on warped wavelets and is adaptive in the regularity
of the model. The convergence of the mean square error to zero, when the
number of simulations of the model tend to infinity, is computed and an
elbow effect is shown, depending on the regularity of the model.
Applications in Epidemiology are carried to illustrate the use of
non-parametric estimators.
\end{abstract}


\section{Introduction}

Sensitivity analysis is widely used for modelling studies in public
health, since the number of parameters involved is often high (see e.g.
\cite{thabane,wudhingragambhir} and references therein). It can be
applied to a variety of problems, and we focus here on the question of
evaluating the impact of input parameters on an output of a model. If
we assume that the output of the model, $y\in \mathbb{R}$, depends on
$p\in \mathbb{N}$ input parameters $x=(x_{1},...x_{p}) \in \mathbb{R}
^{p}$ through the relation $y=f(x)$, we are interested here in
evaluating how the parameter $x_{\ell }$, for $\ell \in \{1,\dots ,p
\}$ affects $y$. The vector $x$ of the input parameters can be
considered as a realisation of a set of random variables $X=(X_{1},...X
_{p})$, with a known distribution and with possibly correlated
components. Also, sensitivity analyses in epidemiology deal with
deterministic models although in many cases, randomness and nuisance
parameters have to be included, which is one of the goal of the present
paper.%

In public health, most of the studies on sensitivity analysis are
performed by letting the input parameters vary on a deterministic grid,
or by sampling all parameters from a prior probability distribution
\cite{Briggs}. However, there exist other ways of measuring the
influence of the inputs on the output. In this article, we are
interested in Sobol indices \cite{sobol}, which are based on an
ANOVA decomposition (see
\cite{saltellilivre2008,jacquesthesis,jacques} for a review). Denoting
by $Y = f(X)$ the random response, the first order Sobol indices can be
defined for $\ell \in \{1,\dots , p\}$ by
%
\begin{equation}
S_{\ell }=\frac{\mbox{Var}\big (\mathbb{E}[Y\ |\ X_{\ell }]\big )}{
\mbox{Var}(Y)}.
\label{def:Sobol1}
\end{equation}
This index represents the fraction of the variance of the output $Y$ due
to the input $X_{\ell }$. Several numerical procedures to estimate the
Sobol indices have been proposed, in particular by Jansen
\cite{jansen} (see also
\cite{saltellilivre2000,saltellilivre2008}). These estimators, that we
recall in the sequel, are based on Monte-Carlo simulations of
$(Y,X_{1}\dots X_{p})$.

The literature focuses on deterministic relations between the input and
output parameters. In a stochastic framework where the model response
$Y$ is not unique for given input parameters, few works have been done,
randomness being usually limited to input variables. Assume that:
%
\begin{equation}
Y=f(X,\varepsilon ),
\label{modelsto}
\end{equation}
where $X=(X_{1},\dots X_{p})$ still denotes the random variables
modelling the uncertainty of the input parameters and where
$\varepsilon $ is a noise variable. In this paper, we will assume that $f$ (and hence $Y$) is bounded by $M>0$. When noise is added in the model,
the classical estimators do not always work: $Y$ can be very sensitive
to the addition of $\varepsilon $. Moreover, this variable is not always
controllable by the user.

When the function $f$ is linear, we can refer to
\cite{fortkleinlagnouxlaurent}. In the literature, meta-models are used:
approximating the mean and the dispersion of the response by
deterministic functions allows to come back to the classical
deterministic framework (e.g. Janon et al.
\cite{janonnodetprieur}, Marrel et al.
\cite{marrelioossdaveigaribatet}). We study here another point of view,
which is based on the non-parametric statistical estimation of the term
$\mbox{Var}\big (\mathbb{E}[Y\ |\ X_{\ell }]\big )$ appearing in the
numerator of \eqref{def:Sobol1}. Approaches based on the Nadaraya-Watson
kernel estimator have been proposed by Da Veiga and Gamboa
\cite{daveiga-gamboa} or Sol\'{\i }s \cite{solisthesis}. We propose
here a new approach based on warped wavelet decompositions introduced
by Kerkyacharian and Picard \cite{kerkyacharianpicard}. An advantage
of these non-parametric estimators is that their computation requires
less simulations of the model. For Jansen estimators, the number of
calls of $f$ required to compute the sensitivity indices is
$n(p+1)$, where $n$ is the number of independent random vectors
$(Y^{i},X^{i}_{1},\dots X^{i}_{p})$ ($i\in \{1,\dots n\}$) that are
sampled for the Monte-Carlo procedure, making the estimation of the
sensitivity indices time-consuming for sophisticated models with many
parameters.%

In Section \ref{section:nonparam}, we present the non-parametric
estimators of the Sobol indices of order 1 in the case of the stochastic
model \eqref{modelsto} and study their convergence rates. The
approximation of $\mbox{Var}\big (\mathbb{E}[Y\ |\ X_{\ell }]\big )$ is
very important to obtain the speed of convergence. When the conditional
expectation is estimated by a Nadaraya-Watson kernel estimator, these
results have been obtained by Sol\'{\i }s \cite{solisthesis} and Da
Veiga and Gamboa \cite{daveiga-gamboa}. The use of wavelets for
estimating the conditional expectation in Sobol indices is new to our
knowledge. Wavelet estimators are more tractable than kernel estimators
in that we do not have to handle approximations of quotients. We derive
the convergence rate for the estimator based on wavelets, using ideas
due to Laurent and Massart \cite{laurentmassart} who considered
estimation of quadratic functionals in a Gaussian setting. Because we
are not necessarily in a Gaussian setting here, we rely on empirical
processes and use sophisticated technology developed by Castellan
\cite{castellan}. Contrarily to the kernel estimators for which
convergence rates rely on assumptions on the joint distribution of
$Y$ and of $X_{1},\dots X_{p}$, we have an upper-bound for the
convergence rates that depend on the regularity of the output $Y$ with
respect to the inputs $X_{1},\dots X_{p}$. Moreover, our estimator is
adaptive and the exact regularity does not need to be known to calibrate
our non-parametric wavelet estimator. Since we estimate covariance
terms, we obtain elbow effects: there is a threshold in the regularity
defining two different regimes with different speeds of convergence for
the estimator. In our case, this allows us to recover convergence rates
in $1/n$ when the model exhibits sufficient regularities. Further
discussion is carried in the body of the article. These estimators are
then computed and compared for toy examples introduced by Ishigami
\cite{ishigamihomma}.

In Section \ref{section:epidemio}, we then address models from
Epidemiology for which non-parametric Sobol estimators have never been
used to our knowledge. First, the stochastic continuous-time SIR model
is considered, in which the population of size $N$ is divided into three
compartments: the susceptibles, infectious and removed individuals (see
e.g. \cite{andersonbritton} for an introduction). Infections and
removals occur at random times whose laws depend on the composition of
the population and on the infection and removal parameters $\lambda $
and $\mu $ as input variables. The output variable $Y$ can be the
prevalence or the incidence at a given time $T$ for instance. $Y$
naturally depends on $\lambda $, $\mu $ and on the randomness underlying
the occurrence of random times. Second, we consider a stochastic
multi-level epidemic model for the transmission of Hepatitis C virus
(HCV) among people who inject drugs (PWID) that has been introduced by
Cousien et al. \cite{cousien1,cousien2}. This model describes an
individual-based population of PWID that is structured by compartments
showing the state of individuals in the heath-care system and by a
contact-graph indicating who injects with whom. Additionally the advance
of HCV in each patient is also taken into account. The input variables
are the different parameters of the model. Ouputs depend on these
inputs, on the randomness of event occurrences and on the randomness of
the social graph. We compare the sensitivity analysis performed by
estimating the Sobol indices of order 1 with the naive sensitivity
analysis performed in \cite{cousien1,cousien2} by letting the
parameters vary in an \textit{a priori} chosen windows.%

In the sequel, $C$ denotes a constant that can vary from line to line.

\section{A non-parametric estimator of the Sobol indices of order 1}%
\label{section:nonparam}

Denoting by $V_{\ell }=\mathbb{E}\big (\mathbb{E}^{2}(Y\ |\ X_{\ell })
\big )$, we have:
%
\begin{equation}
S_{\ell }= \frac{V_{\ell }-\mathbb{E}(Y)^{2}}{\mbox{Var}(Y)},
\label{def:S_ell}
\end{equation}
which can be approximated by
%
\begin{equation}
\widehat{S}_{\ell }=\frac{\widehat{V}_{\ell }- \bar{Y}^{2}}{
\widehat{\sigma }^{2}_{Y}}
\label{def:generique}
\end{equation}
where
\begin{equation*}
\bar{Y}=\frac{1}{n}\sum _{j=1}^{n} Y_{j}\mbox{ and }\widehat{\sigma }
^{2}_{Y}=\frac{1}{n} \sum _{j=1}^{n} (Y_{j}-\bar{Y})^{2}
\end{equation*}
are the empirical mean and variance of $Y$. We can think of several
approximations $\widehat{V}_{\ell }$ of $V_{\ell }$, for example, based
on Nadaraya-Watson and on warped wavelet estimators. At an advanced
stage of this work, we learned that the Nadaraya-Watson-based estimator
of Sobol indices of order 1 had also been proposed and studied in the
PhD of Sol\'{i}s \cite{solisthesis}. Using a result on estimation
of covariances by Loubes et al. \cite{loubesmarteausolis}, they
obtain an elbow effect. However their estimation is not adaptive, and
requires the knowledge of the regularity of the joint density function
of $(X_{\ell },Y)$. For the warped wavelet estimator, we propose a model
selection procedure based on a work by Laurent and Massart
\cite{laurentmassart} to make the estimator adaptive.

\subsection{Definitions}%
\label{section:def}

Assume that we have $n$ independent couples $(Y^{i}, X_{1}^{i},\dots
X_{p}^{i})$ in $\mathbb{R}\times \mathbb{R}^{p}$, for $i\in \{1,
\dots , n\}$, generated by \eqref{modelsto}.

Our wavelet estimator is based on a warped wavelet decomposition of\break
$\mathbb{E}(Y\, |\, X_{\ell }=x)$. Let us denote by $L^{2}(\mu )$ the
space of real functions that are square integrable with respect to the
measure $\mu $. When we do not specify $\mu $, $L^{2}$ denotes the space
of real functions that are square integrable with respect to the
Lebesgue measure on $\mathbb{R}$. In the sequel, we denote by
$\langle f,g\rangle =\int _{\mathbb{R}}f(u)g(u)du$, for $f,g\in L^{2}$,
the usual scalar product of $L^{2}$. The associated $L^{2}$-norm is
$\|f\|_{2}^{2}=\int _{\mathbb{R}}f^{2}(u)du$. Wavelet estimators are
projection estimators, and $L^{2}$ is a natural setting to work with.
But when dealing with a probability framework, one can face the need to
consider different Hilbert structures. Let now $\mu $ be a probability
measure with cumulative distribution function $G$. Warped wavelet
decompositions introduced by Kerkyacharian and Picard
\cite{kerkyacharianpicard} allow, in a very natural way, to consider
wavelet decompositions in $L^{2}(\mu )$: composing any Hilbert basis of
$L^{2}$ by $G$ provides a Hilbert basis of $L^{2}(\mu )$. See
\cite{chagny,kerkyacharianpicard} for more details.
\par
Let us denote by $G_{\ell }$ the cumulative distribution function of
$X_{\ell }$ and let $(\psi _{jk})_{j\geq -1,k\in \mathbb{Z}}$ be a
Hilbert wavelet basis of $L^{2}$. The wavelet $\psi _{-10}$ is the father
wavelet, and for $k\in \mathbb{Z}$, $\psi _{-1k}(x)=\psi _{-10}(x-k)$. The
wavelet $\psi _{00}$ is the mother wavelet, and for $j\geq 0$,
$k\in \mathbb{Z}$, $\psi _{jk}(x)=2^{j/2} \psi _{00}(2^{j} x -k)$. In the
sequel, we will consider wavelets with compact support. The warped
wavelet basis that we will consider is $(\psi _{jk}\circ G)_{j\geq -1,k
\in \mathbb{Z}}$.

\begin{definition}
\label{def:estimateur2}%
Let us define for $j\geq -1$, $k\in \mathbb{Z}$,
%
\begin{equation}
\widehat{\beta }^{\ell }_{jk}=\frac{1}{n} \sum _{i=1}^{n} Y_{i} \psi
_{jk}(G_{\ell }(X_{\ell }^{i})).
\end{equation}
Then, we define the (block thresholding) estimator of $S_{\ell }$ as
\eqref{def:generique} with
%
\begin{align}
\widehat{V}_{\ell }=\sum _{j= -1}^{J_{n}} \Big [ \sum _{k\in \mathbb{Z}}
\big (\widehat{\beta }^{\ell }_{jk}\big )^{2} - w(j)\Big ]\ind _{
\sum _{k\in \mathbb{Z}}\big (\widehat{\beta }^{\ell }_{jk}\big )^{2}
\geq w(j)},
\label{def:hat_theta}
\end{align}
where $w(j)=K \Big (\frac{2^{j}+\log 2}{n}\Big )$, $J_{n}:=\big [\log
_{2}\big (\sqrt{n}\big )\big ]$ (with $[.]$ denoting the integer part),
and $K$ is a positive constant.
\end{definition}

Let us present the idea explaining the estimator proposed in Definition
\ref{def:estimateur2}. Let us introduce centered random variables
$\eta _{\ell }$ such that
%
\begin{equation}\label{eq:7}
Y=f(X,\varepsilon )=\mathbb{E}(Y\, |\, X_{\ell })+\eta _{\ell }.
\end{equation}
Let $g_{\ell }(x)=\mathbb{E}(Y\, |\, X_{\ell }=x)$ and $h_{\ell }(u)=g
_{\ell }\circ G_{\ell }^{-1}(u)$. $h_{\ell }$ is a function from
$[0,1]\mapsto \mathbb{R}$ that belong to $L^{2}$ since $Y\in L^{2}$.
Then
%
\begin{align}\label{eq:8}
& h_{\ell }(u)=\sum _{j\geq -1} \sum _{k\in \mathbb{Z}} \beta _{jk}^{
\ell }\psi _{jk}(u),
\\
& \mbox{ with }\quad \beta _{jk}^{\ell }=\int _{0}^{1} h_{\ell }(u)
\psi _{jk}(u)du=\int _{\mathbb{R}}g_{\ell }(x) \psi _{jk}(G_{\ell }(x)) G
_{\ell }(dx).
\nonumber
\end{align}
Notice that the sum in $k$ is finite because the function $h_{\ell }$
has compact support in $[0,1]$. It is then natural to estimate
$h_{\ell }(u)$ by
%
\begin{equation}
\widehat{h}_{\ell }=\sum _{j\geq -1}\sum _{k\in \mathbb{Z}}
\widehat{\beta }_{jk}^{\ell }\psi _{jk}(u).
\label{def:hat_h}
\end{equation}
We can then rewrite $V_{\ell}$ as:
%
\begin{align}
V_{\ell }=\mathbb{E}\big (\mathbb{E}^{2}(Y\, |\, X_{\ell })\big )=
&
\int _{\mathbb{R}} G_{\ell }(dx)\Big (\sum _{j\geq -1} \sum _{k\in
\mathbb{Z}} \beta _{jk}^{\ell }\psi _{jk}\big (G_{\ell }(x)\big )\Big )^{2}
\nonumber
\\
=
& \int _{0}^{1} \Big (\sum _{j\geq -1} \sum _{k\in \mathbb{Z}} \beta
_{jk}^{\ell }\psi _{jk}(u)\Big )^{2} \ du
\nonumber
\\
=
& \sum _{j\geq -1}\sum _{k\in \mathbb{Z}} \big ( \beta _{jk}^{\ell }
\big )^{2}=\|h_{\ell }\|^{2}_{2}.
\end{align}
Adaptive estimation of $\|h_{\ell }\|^{2}_{2}$ has been studied in
\cite{laurentmassart}, which provides the block thresholding estimator
$\widehat{V}_{\ell }$ in Definition \ref{def:estimateur2}. The idea is:
1) to sum the terms $\big ( \beta _{jk}^{\ell }\big )^{2}$, for
$j\geq 0$, by blocks $\{(j,k),\ k\in \mathbb{Z}\}$ for $j\in \{-1,
\dots ,J_{n} \}$ with a penalty $w(j)$ for each block to avoid choosing
too large $j$s, 2) to cut the blocks that do not sufficiently contribute
to the sum, in order to obtain statistical adaptation.%

Notice that $\widehat{V}_{\ell }$ can be seen as an estimator of
$V_{\ell }$ resulting from a model selection on the choice of the blocks
$\{(j,k),\ k\in \mathbb{Z}\}$, $j\in \{-1,\dots , J_{n}\}$ that are
kept, with the penalty function
%
\begin{equation}
\label{def:penalty}
\mbox{pen}(\mathcal{J})=\sum _{j\in \mathcal{J}} w(j),
\end{equation}
for $\mathcal{J}\subset \{-1,\dots , J_{n}\}$. Indeed:
%
\begin{align}
\widehat{V}_{\ell }=
& \sup _{\mathcal{J}\subset \{-1,0,\dots ,J_{n}\}}
\sum _{j\in \mathcal{J}} \Big [ \sum _{k\in \mathbb{N}}\big (
\widehat{\beta }^{\ell }_{jk}\big )^{2} - w(j)\Big ]
\nonumber
\\
=
& \sup _{\mathcal{J}\subset \{-1,0,\dots ,J_{n}\}}
\sum _{j\in \mathcal{J}} \sum _{k\in \mathbb{N}}\big (\widehat{\beta }
^{\ell }_{jk}\big )^{2} - \mbox{pen}(\mathcal{J}).
\label{lien_pen}
\end{align}

Remark that the definition of the estimator and the penalization depend
on a constant $K$ through the definition of $w(j)$. The value of this
constant is chosen in order to obtain oracle inequalities. In practice,
this constant is hard to compute, and can be chosen by a slope heuristic
approach (see e.g. \cite{arlotmassart}).

\subsection{Statistical properties}

In this Section, we are interested in the rate of convergence to zero
of the mean square error (MSE) $\mathbb{E}\big ((S_{\ell }-\widehat{S}
_{\ell })^{2}\big )$.

\begin{lemma}
\label{lemme1}%
Consider the generic estimator $\widehat{S}_{\ell }$ defined in
\eqref{def:generique}, where $\widehat{V}_{\ell }$ is any estimator of
$V_{\ell }=\mathbb{E}(\mathbb{E}^{2}(Y\ |\ X_{\ell }))$. Then there is
a constant $C$ and an integer $n_{0}$ such that for all $n\geq n_{0}$,
%
\begin{equation}
\mathbb{E}\big ((S_{\ell }-\widehat{S}_{\ell })^{2}\big ) \leq
\frac{C}{n}+\frac{4}{\mbox{Var}(Y)^{2}} \mathbb{E}\Big [\big (
\widehat{V}_{\ell }- V_{\ell }\big )^{2}\Big ].
\label{erreur:generique}
\end{equation}
\end{lemma}

\begin{proof}
From \eqref{def:S_ell} and \eqref{def:generique},
%
\begin{align}
\mathbb{E}\big ((S_{\ell }-\widehat{S}_{\ell })^{2}\big )=
& \mathbb{E}
\Big [\Big (\frac{V_{\ell }-\mathbb{E}(Y)^{2}}{\mbox{Var}(Y)}-\frac{
\widehat{V}_{\ell }-\bar{Y}^{2}}{\widehat{\sigma }_{Y}^{2}}\Big )^{2}
\Big ]
\nonumber
\\
\leq
& 2 \mathbb{E}\Big [\Big (\frac{\mathbb{E}(Y)^{2}}{\mbox{Var}(Y)}-\frac{
\bar{Y}^{2}}{\widehat{\sigma }_{Y}^{2}}\Big )^{2}\Big ]+2\mathbb{E}
\Big [\Big (\frac{V_{\ell }}{\mbox{Var}(Y)}-\frac{\widehat{V}_{\ell }}{
\widehat{\sigma }_{Y}^{2}}\Big )^{2}\Big ].
\label{etape1}
\end{align}
The first term in the right hand side (r.h.s.) is in $C/n$ for
sufficiently large $n$. For the second term in the right hand side of
\eqref{etape1}:
%
\begin{equation}
\mathbb{E}\Big [\Big (\frac{V_{\ell }}{\mbox{Var}(Y)}-\frac{\widehat{V}
_{\ell }}{\widehat{\sigma }_{Y}^{2}}\Big )^{2}\Big ]
\leq 2 \mathbb{E}\Big [\widehat{V}_{\ell }^{2} \Big (\frac{1}{
\mbox{Var}(Y)}-\frac{1}{\widehat{\sigma }_{Y}^{2}}\Big )^{2}\Big ]+\frac{2}{
\mbox{Var}(Y)^{2}} \mathbb{E}\Big [\big (\widehat{V}_{\ell }- V_{\ell }
\big )^{2}\Big ].
\label{etape2}
\end{equation}
The first term in the r.h.s. is also in $C/n$, which concludes the
proof.
\end{proof}

When $\widehat{V}_{\ell }$ is a Nadaraya-Watson estimator, Loubes et al.
\cite{loubesmarteausolis} established from Lemma \ref{lemme1} a
control of the MSE that looks like the result we announce and comment
in Corollary \ref{corol:vitesseconvergence}. Their result is based on
\eqref{erreur:generique} and a bound for $\mathbb{E}\Big [\big (
\widehat{V}_{\ell }- V_{\ell }\big )^{2}\Big ]$ given by
\cite[Th. 1]{loubesmarteausolis}, whose proof is technical. Here, we
consider the estimator $\widehat{V}_{\ell }$ introduced in
\eqref{def:hat_theta} and upper-bound the MSE. Our proof is much shorter
than theirs, due to the nature of the estimators and to the techniques
that we use.%

Let us introduce first some additional notation. For $\mathcal{J}
\subset \{-1,\dots ,J_{n}\}$, we define the projection $h_{\mathcal{J,
\ell }}$ of $h$ on the subspace spanned by $\{\psi _{jk},\mbox{ with }
j\in \mathcal{J},\, k\in \mathbb{Z}\}$ and its estimator $\widehat{h}
_{\mathcal{J},\ell }$:
%
\begin{align}
& h_{\mathcal{J},\ell }(u)=\sum _{j\in \mathcal{J}} \sum _{k\in
\mathbb{Z}} \beta ^{\ell }_{jk}\psi _{jk}(u)
\\
& \widehat{h}_{\mathcal{J,\ell }}(u)=\sum _{j\in \mathcal{J}}
\sum _{k\in \mathbb{Z}} \widehat{\beta }^{\ell }_{jk}\psi _{jk}(u).
\end{align}
We also introduce the estimator of $V_{\ell }$ for a fixed subset of
resolutions $\mathcal{J}$:
%
\begin{align}
\widehat{V}_{\mathcal{J},\ell }= \| \widehat{h}_{\mathcal{J},\ell }\|
^{2}_{2}=\sum _{j\in \mathcal{J}} \sum _{k\in \mathbb{Z}}\big (
\widehat{\beta }^{\ell }_{jk}\big )^{2}.
\label{def:est_interm2}
\end{align}
Note that $\widehat{V}_{\mathcal{J},\ell }$ is one possible estimator
$\widehat{V}_{\ell }$ in Lemma \ref{lemme1}.%

The estimators $\widehat{\beta }_{jk}$ and $\widehat{V}_{\mathcal{J},
\ell }$ have natural expressions in term of the empirical process
$\gamma _{n}(dx)$ defined as follows:
%
\begin{definition}
\label{def:gamma_n}%
The empirical measure associated with our problem is:
%
\begin{equation}
\gamma _{n}(dx)=\frac{1}{n}\sum _{i=1}^{n} Y_{i}
\delta _{G_{\ell }(X_{\ell }^{i})}(dx)
\end{equation}
where $\delta _{a}(dx)$ denotes the Dirac mass in $a$.%

For a measurable function $f$, $ \gamma _{n}(f)=\frac{1}{n}\sum _{i=1}
^{n} Y_{i} f\big (G_{\ell }(X_{\ell }^{i})\big ) $. We also define the
centered integral of $f$ with respect to $\gamma _{n}(dx)$ as:
%
\begin{align}
\bar{\gamma }_{n}(f)=
& \gamma _{n}(f)-\mathbb{E}\big (\gamma _{n}(f)
\big )
\label{def:gamma_bar}%
\\
=
& \frac{1}{n}\sum _{i=1}^{n} \Big (Y_{i} f\big (G_{\ell }(X_{\ell }
^{i})\big )-\mathbb{E}\big [Y_{i}f\big (G_{\ell }(X_{\ell }^{i})\big )
\big ]\Big ).
\nonumber
\end{align}
\end{definition}

Using the empirical measure $\gamma _{n}(dx)$, we have:
\begin{align*}
\widehat{\beta }^{\ell }_{jk}=\gamma _{n}\big (\psi _{jk}\big )=\beta ^{
\ell }_{jk}+\bar{\gamma }_{n}\big (\psi _{jk}\big ).
\end{align*}
Let us also introduce the correction term using \eqref{eq:7}, \eqref{eq:8} and \eqref{def:gamma_bar}:
%
\begin{align}
\zeta _{n}=
& 2 \bar{\gamma }_{n}\big (h_{\ell }\big )
\label{def:zetan}%
\\
=
& 2 \Big [\frac{1}{n}\sum _{i=1}^{n} Y_{i} h_{\ell }\big (G_{\ell }(X
_{\ell }^{i})\big ) - \mathbb{E}\Big (Y_{1} h_{\ell }\big (G_{\ell }(X
_{\ell }^{1})\big )\Big )\Big ]
\nonumber
\\
=
& 2 \Big [\frac{1}{n}\sum _{i=1}^{n} h^{2}_{\ell }\big (G_{\ell }(X
_{\ell }^{i})\big )- \|h_{\ell }\|_{2}^{2} \Big ]+\frac{2}{n} \sum _{i=1}
^{n} \eta _{\ell }^{i} h_{\ell }\big (G_{\ell }(X_{\ell }^{i})\big ).
\label{etape:zeta}
\end{align}

\begin{theorem}
\label{th1}
Let us assume that the random variables $Y$ are bounded by a constant
$M>0$, and let us choose a father and a mother wavelets $\psi _{-10}$ and
$\psi _{00}$ that are continuous with compact support (and thus bounded).
The estimator $\widehat{V}_{\ell }$ defined in \eqref{def:hat_theta} is
almost surely finite, and:
%
\begin{equation}
\mathbb{E}\Big [\big (\widehat{V}_{\ell }-V_{\ell }- \zeta _{n} \big )^{2}
\Big ]
\leq C\ \inf _{\mathcal{J}\subset \{-1,\dots ,J_{n}\}} \Big (\|h_{
\ell }-h_{\mathcal{J},\ell }\|^{4}_{2} + \frac{\mbox{Card}^{2}(
\mathcal{J})}{n^{2}}\Big )+\frac{C'\log _{2}^{2}(n)}{n^{3/2}},
\label{oracle}
\end{equation}
for constants $C$ and $C'>0$.
\end{theorem}

We deduce the following corollary from the estimate obtained above. Let
us consider the Besov space $\mathcal{B}(\alpha ,2,\infty )$ of
functions $h=\sum _{j\geq -1}\sum _{k\in \mathbb{Z}} \beta _{jk} \psi
_{jk}$ of $L^{2}$ such that
\begin{equation*}
|h|_{\alpha ,2,\infty }:= \sum _{j\geq 0} 2^{j\alpha }\sqrt{
\sup _{0<v\leq 2^{-j}} \int _{0}^{1-v}|h(u+v)-h(u)|^{2}du} <+\infty .
\end{equation*}

For a $h\in \mathcal{B}(\alpha ,2,\infty )$ and for its projection
$h_{\mathcal{J}}$ on $\mbox{Vect}\{\psi _{jk},\ j\in \mathcal{J}=\{-1,
\dots J_{\mbox{{\scriptsize max}}}\},\ k\in \mathbb{Z}\}$ (with
$J_{\mbox{{\scriptsize max}}}=\max \mathcal{J}$), we have the following
approximation result from \cite[Th. 9.4]{hardlekerkapicardtsyba}.
%
\begin{prop}[H\"{a}rdle Kerkyacharian Picard and Tsybakov]
\label{prop:besov}
Assume that the wavelet function $\psi _{-10}$ has compact support and
is of class $\mathcal{C}^{N}$ for an integer $N>0$. Then, if
$h\in \mathcal{B}(\alpha ,2,\infty )$ with $\alpha <N+1$,
%
\begin{equation}
\sup _{\mathcal{J}\subset \mathbb{N}\cup \{-1\}} 2^{\alpha J_{
\mbox{{\scriptsize max}}}} \|h-h_{\mathcal{J}}\|_{2}
= \sup _{\mathcal{J}\subset \mathbb{N}\cup \{-1\}} 2^{\alpha J_{
\mbox{{\scriptsize max}}}} \Big ( \sum _{j\geq J_{
\mbox{{\scriptsize max}}}} \sum _{k\in \mathbb{Z}} \beta _{jk}^{2}
\Big )^{1/2} <+\infty .
\end{equation}
\end{prop}
Notice that Theorem 9.4 of \cite{hardlekerkapicardtsyba} requires
assumptions that are fulfilled when $\psi _{-10}$ has compact support and
is smooth enough (see the comment after the Corol. 8.2 of
\cite{hardlekerkapicardtsyba}).

\begin{corollary}
\label{corol:vitesseconvergence}%
If $\psi _{-10}$ has compact support and is of class $\mathcal{C}^{N}$
for an integer $N>0$ and if $h_{\ell }$ belongs to a ball of radius
$R>0$ of $\mathcal{B}(\alpha ,2,\infty )$ for $0<\alpha <N+1$, then
%
\begin{align}
\sup _{h\in \mathcal{B}(\alpha ,2,\infty )}\mathbb{E}\Big [\big (
\widehat{V}_{\ell }-V_{\ell }\big )^{2}\Big ]\leq
& C \Big (n^{-\frac{8
\alpha }{4\alpha +1}}+\frac{1}{n}\Big ).
\end{align}
As a consequence, we obtain the following elbow effect:%

If $\alpha \geq \frac{1}{4}$, there exists a constant $C>0$ such that
\begin{equation*}
\mathbb{E}\big ((S_{\ell }-\widehat{S}_{\ell })^{2}\big )\leq
\frac{C}{n}.
\end{equation*}
If $\alpha <\frac{1}{4}$, there exists a constant $C>0$ such that
\begin{equation*}
\mathbb{E}\big ((S_{\ell }-\widehat{S}_{\ell })^{2}\big )\leq C n^{-\frac{8
\alpha }{4\alpha +1}}.
\end{equation*}
\end{corollary}

The proof of Theorem \ref{th1} is postponed to Section
\ref{section:proof_th}. Let us remark that in comparison with the result
of Loubes et al. \cite{loubesmarteausolis}, the regularity
assumption is on the function $h_{\ell }$ rather than on the joint
density $\phi (x,y)$ of $(X_{\ell }, Y)$. The adaptivity of our
estimator is then welcomed since the function $h_{\ell }$ is \textit{a
priori} unknown. Remark that in application, the joint density
$\phi (x,y)$ also has to be estimated and hence has an unknown
regularity.

When $\alpha <1/4$ and $\alpha \rightarrow 1/4$, the exponent
$8\alpha /(4\alpha +1)\rightarrow 1$. In the case when $\alpha > 1/4$,
we can show from the estimate of Th. \ref{th1} that:
%
\begin{equation}
\lim _{n\rightarrow +\infty } n \mathbb{E}\Big [\big (\widehat{V}_{
\ell }-V_{\ell }- \zeta _{n} \big )^{2}\Big ]=0,
\end{equation}
which yields that $\sqrt{n}\big (\widehat{V}_{\ell }-V_{\ell }- \zeta
_{n} \big )$ converges to 0 in $L^{2}$. Since $\sqrt{n}\zeta _{n}$
converges in distribution to $\mathcal{N}\Big (0,4 \mbox{Var}\big (Y
_{1} h_{\ell }(G_{\ell }(X^{1}_{\ell }))\big )\Big )$ by the central limit
theorem, we obtain that:
%
\begin{equation}
\lim _{n\rightarrow +\infty }\sqrt{n}\big (\widehat{V}_{\ell }-V_{
\ell }\big )=\mathcal{N}\Big (0,4 \mbox{Var}\big (Y_{1} h_{\ell }(G_{
\ell }(X^{1}_{\ell }))\big )\Big ),
\end{equation}
in distribution.

The result of Corollary \ref{corol:vitesseconvergence} is stated for
functions $h_{\ell }$ belonging to $\mathcal{B}(\alpha ,2,\infty )$, but
the generalization to other Besov spaces might be possible.

\subsection{Numerical tests on toy models}%
\label{section:ishigami}

We start with considering toy models based on the Ishigami function,
often chosen as benchmark:
%
\begin{equation}
Y=f(X_{1}, X_{2}, X_{3})=\sin (X_{1})+7\sin (X_{2})^{2}+0.1\ X_{3}
^{4} \sin (X_{1})
\end{equation}
where $X_{i}$ are independent uniform random variables in $[-\pi ,
\pi ]$ (see e.g. \cite{ishigamihomma,saltellilivre2000}).%

\medskip\noindent
\textit{Case 1 -- Ishigami model}: first, we consider this model with
$(X_{1}, X_{2}, X_{3})$ as input parameters and compute the associated
Sobol indices. For the Ishigami function, all the Sobol sensitivity
indices are known.
\begin{equation*}
S_{1}=0.3139,
\qquad
S_{2}=0.4424,
\qquad
S_{3}=0.
\end{equation*}
\textit{Case 2 -- stochastic Ishigami model}: following Marrel et al.
\cite{marrelioossdaveigaribatet}, we consider the case where
$(X_{1},X_{2})$ are the input parameters and $X_{3}$ a nuisance random
parameter. The Sobol indices relative to $X_{1}$ and $X_{2}$ have the
same values as in the first case.%

In each case, we compare the estimator of the Sobol indices of order 1
based on the wavelet regressions with two other estimators:
\begin{itemize}%
\item
the Jansen estimator, which is one of the classical estimator found in
the literature (see
\cite{jansen,saltelliannoniazzinicampolongorattotarantola} for Jansen
and other estimators). This estimator is based on the mixing of two
samples $(X^{(1),i}_{1},...,X^{(1),i}_{p},\ i\in \{1,\dots n\})$ and
$(X^{(2),i}_{1},...,X^{(2),i}_{p},\ i\in \{1,\dots n\})$ of i.i.d.
$p$-uplets distributed as $(X_{1},\dots X_{p})$: for the first order
Sobol indices, $ \forall \ell \in {1,...,p}$:
%
\begin{multline}
\widehat{S}_{\ell }= 1 - \frac{1}{2n \ \widehat{\sigma }_{Y}^{2}}
\sum _{i=1}^{n} \big (f(X^{(2),i}_{1},...,X^{(2),i}_{p})
\\
- f(X^{(1),i}_{1},...,X^{(1),i}_{\ell -1}, X^{(2),i}_{\ell },X^{(1),i}
_{\ell +1},\dots , X^{(1),i}_{p})\big )^{2}.
\label{estimator:jansen}
\end{multline}%
\item
the estimator \eqref{def:generique} defined with the choice of the
Nadaraya-Watson regression estimator for $\widehat{V}_{\ell }$ (e.g.
\cite{tsybakov_livre}) instead of the wavelet estimator
\eqref{def:hat_theta}:
\begin{equation*}
\widehat{V}_{\ell }=\frac{\sum _{j=1}^{n} Y_{j} K_{h}(X_{\ell }^{j}-x)}{
\sum _{j=1}^{n} K_{h}(X_{\ell }^{j}-x)}.
\end{equation*}
This provides the estimator:
%
\begin{equation}
\widehat{S}_{\ell }=\frac{\frac{1}{n}\sum _{i=1}^{n} \Big (\frac{\sum
_{j =1}^{n} Y_{j} K_{h}(X_{\ell }^{j}-X_{\ell }^{i})}{\sum _{j =1}^{n}
K_{h}(X_{\ell }^{j}-X_{\ell }^{i})} \Big )^{2}-\bar{Y}^{2}}{
\widehat{\sigma }_{Y}^{2}}.
\end{equation}
\end{itemize}
Notice that the estimations using Jansen estimators require
$(p+1)n$ calls to $f$, which is in many real cases the most expensive
numerically. To enable comparisons, we compute the non-parametric
estimators of the $S_{\ell }$'s from samples of size $(p+1)n$. We used
$n=10{,}000$. To obtain Monte-Carlo approximations of the estimators'
distributions, we performed 1,000 replications from which we estimate
the bias and MSE for each estimator. For the wavelet (resp.
Nadaraya-Watson) estimator, we choose the constant $K$ (resp. window
$h$) by a leave-one-out cross validation procedure
\cite[Section 1.4]{tsybakov_livre}. For the wavelet estimator, we use
the Daubechies 4 wavelet basis when implementing the wavelet estimator.

\paragraph{Case 1}
The results are presented in Table \ref{tb:ishidet}. When comparing the
MSE, the performances of the Jansen estimators are overall lower than
the non-parametric estimators, but the bias is usually smaller. For
$X_{1}$ and $X_{2}$, the Nadaraya-Watson and wavelet estimators have
comparable performances, but for $X_{3}$ the Nadaraya-Watson estimator
performs better. This is due to the fact that the window $h$ for this
variable can be chosen large since the function to estimate is flat (see
Figure \ref{fig:exampleIshi40000}).

\begin{table}[!ht]
\caption{Estimates of the bias and MSE for the parameters $X_{1}$, $X_{2}$ and $X_{3}$ in the Ishigmami function, for 1,000 replications and $n=10{,}000$}
\label{tb:ishidet}
\centering
\scalebox{0.8}{
\begin{tabular}{| l c c c c c c |}
\hline
Method & $\mathbb{E}[\hat{S_{1}}-S_{1}]$ & $\mathbb{E}[(\hat{S_{1}}-S_{1})^{2}]$ & $\mathbb{E}[\hat{S_{2}}-S_{2}]$ & $\mathbb{E}[(\hat{S_{2}}-S_{2})^{2}$ & $\mathbb{E}[\hat{S_{3}}-S_{3}]$ & $\mathbb{E}[(\hat{S_{3}}-S_{3})^{2}]$ \\
\hline
Jansen & 9,90E-04 & 1,80E-04 & 3,20E-05 & 1,00E-04 & 8,60E-04 & 5,60E-04 \\
Nadaraya-Watson & $-$1,00E-03 & 1,30E-05 & $-$9,90E-03 & 1,10E-04 & $-$1,20E-05& 1,10E-07\\
Wavelets & $-$1,10E-03 & 4,00E-05 & 1,30E-03 & 6,60E-05 & 3,90E-03 & 2,20E-05\\
\hline
\end{tabular}
}\end{table}

\begin{figure}
\centering
\includegraphics[width=7cm]{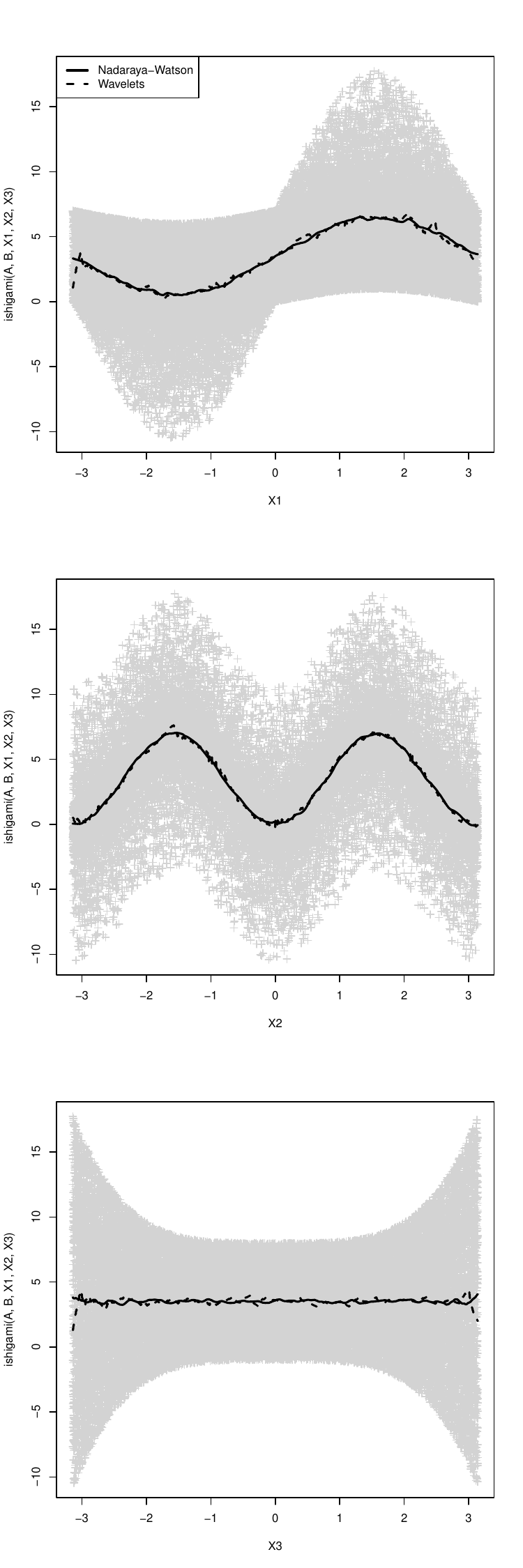}
\caption{{\small \textit{Example of regression obtained using Nadaraya-Watson and wavelets with $n(p+1)=40,000$ simulations for the Ishigami function. The conditional expectation of $Y$ knowing $X_1$ (resp. $X_2$ and $X_3$) is represented in line 1 (resp. 2 and 3).}}}\label{fig:exampleIshi40000}
\end{figure}

\paragraph{Case 2}
The results are presented in Table \ref{tb:ishistoch}. As for Case 1,
we see that in term of MSE, the non-parametric estimators overperform
again the Jansen estimators. For $X_{1}$, the Nadaraya-Watson and
wavelet estimators have comparable statistics, but the wavelet estimator
is the best for $X_{2}$.

\begin{table}[!ht]
\caption{Estimates of the bias and MSE for the parameters $X_{1}$ and
$X_{2}$ in the Ishigmami function, when $X_{3}$ is considered as a
pertubation parameter, for 1,000 replications and $n=10{,}000$}
\label{tb:ishistoch}
\centering
\begin{tabular}{| l c c c c|}
\hline
Method & $\mathbb{E}[\hat{S_{1}}-S_{1}]$ & $\mathbb{E}[(\hat{S_{1}}-S_{1})^{2}]$ & $\mathbb{E}[\hat{S_{2}}-S_{2}]$ & $\mathbb{E}[(\hat{S_{2}}-S_{2})^{2}$ \\
\hline
Jansen & $-$5,60E-04 & 2,00E-04 & $-$7,80E-04 & 1,80E-04\\
Nadaraya-Watson & $-$1,80E-03 & 1,70E-05 & $-$1,40E-02 & 2,00E-04\\
Wavelets & $-$7,00E-04& 5,60E-05 & 1,90E-03 & 8,70E-05\\
\hline
\end{tabular}%
\end{table}

Both non-parametric estimators depend on a tuning parameter: the window
$h$ for Nadaraya-Watson and the constant $K$ for the wavelets. In Figure
\ref{fig:HKvarie}, the MSE are plotted as functions of the window
$h$ (for the estimator with Nadaraya-Watson) and of the constant $K$
(for our estimator with wavelets). The performances of the wavelet
estimator are much more stable with respect to the values of $K$ on the
stochastic Ishigami model (Case 2).

\begin{figure}[!t]
\centering
\begin{tabular}{cc}
\includegraphics[width=7cm]{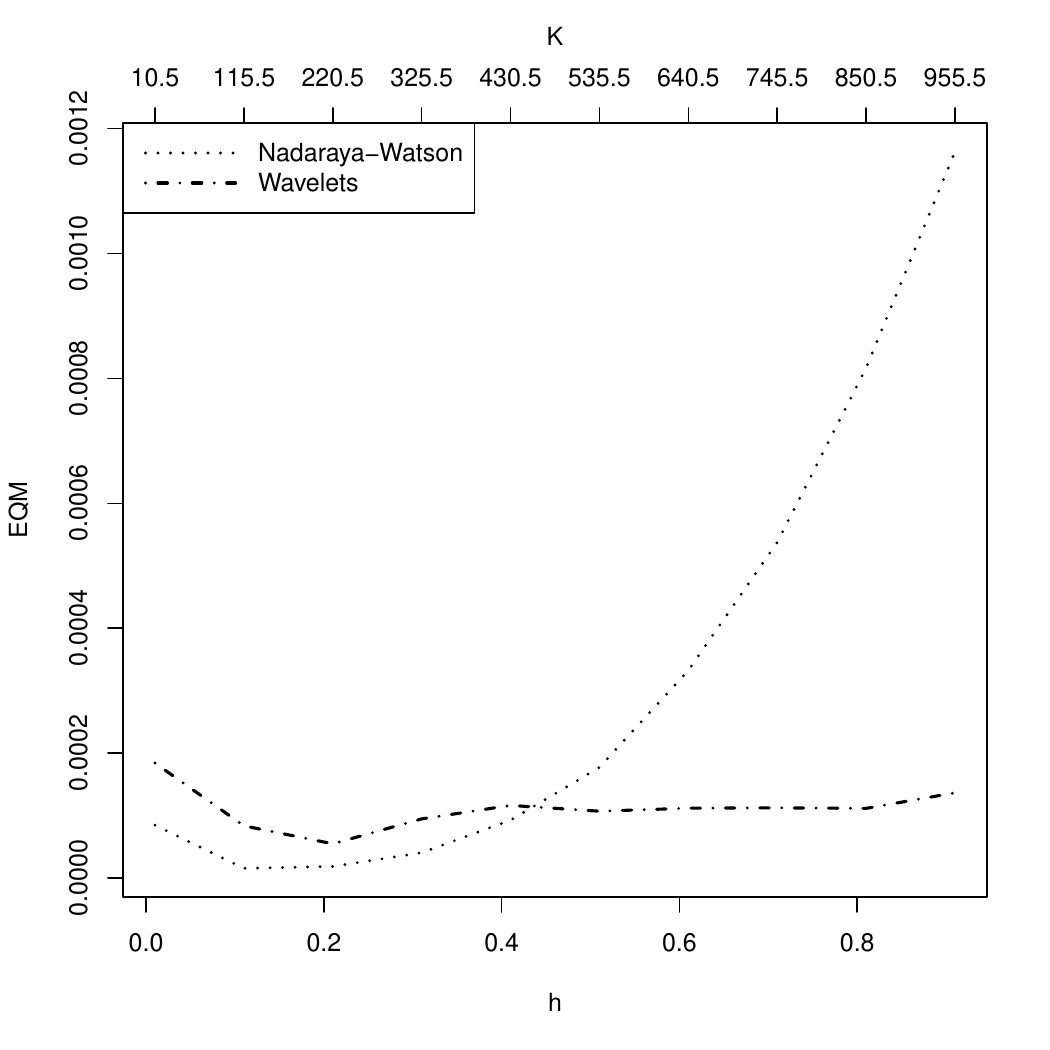} &
\includegraphics[width=7cm]{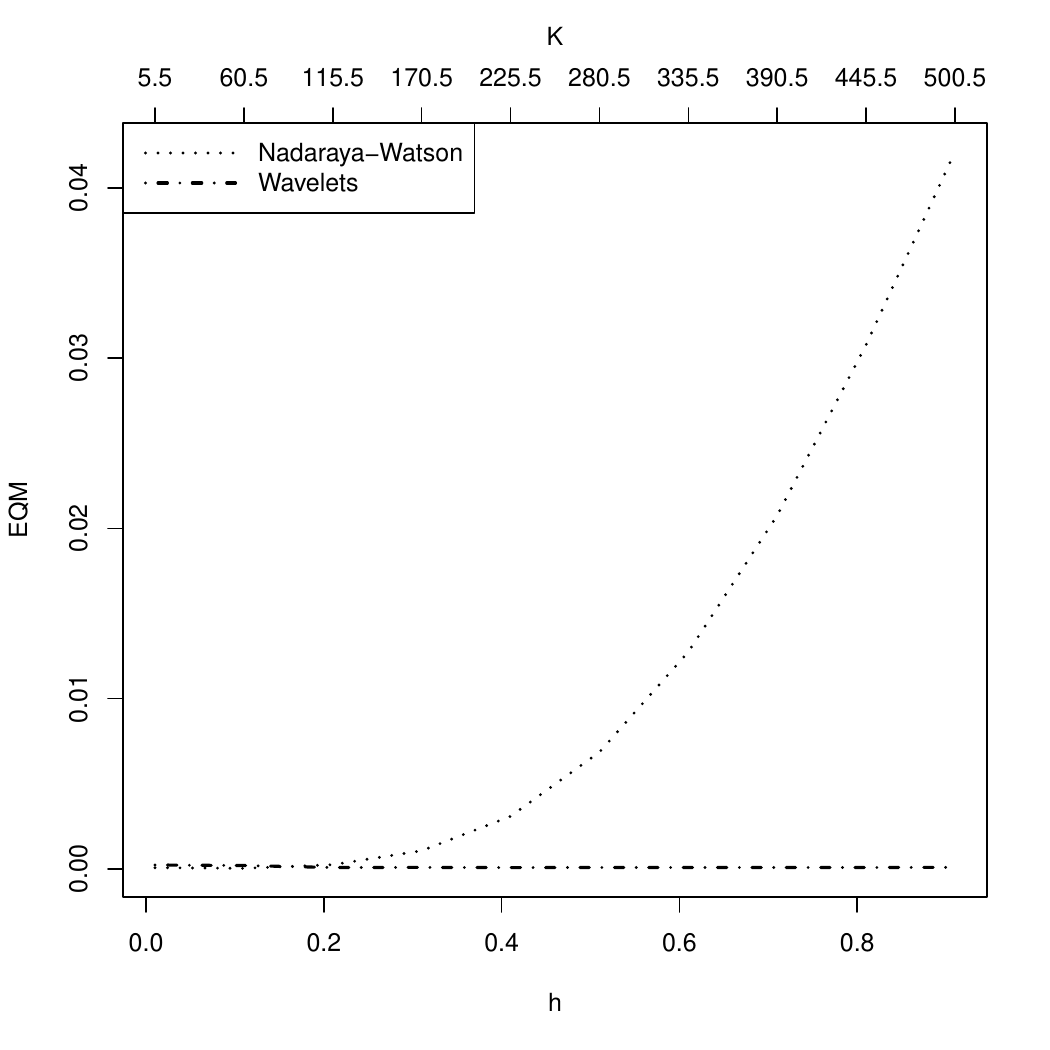}
\end{tabular}
\caption{{\small \textit{MSE for the estimators using Nadaraya-Watson (dots) and wavelets (dash-dots) in the case of the stochastic Ishigami model (Case 2) for $n=10,000$: for $X_1$ (left) and $X_2$ (right). The MSE with Nadaraya-Watson estimators are plotted as functions of the window $h$ (bottom axis) and the MSE with wavelets are plotted as function of the constant $K$ (top axis).  }}}\label{fig:HKvarie}
\end{figure}

To conclude these simulations on the stochastic Ishigami model, we
plotted on logarithmic scales the MSE as function of the sample size
$n$: see Figure \ref{fig:nvarie}. It is seen that the wavelet estimator
is better than the Jansen estimator. For the wavelet estimator, the
slope estimated with ordinary least squares equals to $-1.15$ for
$X_{1}$ and $-1.12$ for $X_{2}$. This is in accordance with the value
of $-$1 predicted by Corollary \ref{corol:vitesseconvergence}.

\begin{figure}[!t]
\centering
\includegraphics[height=7cm]{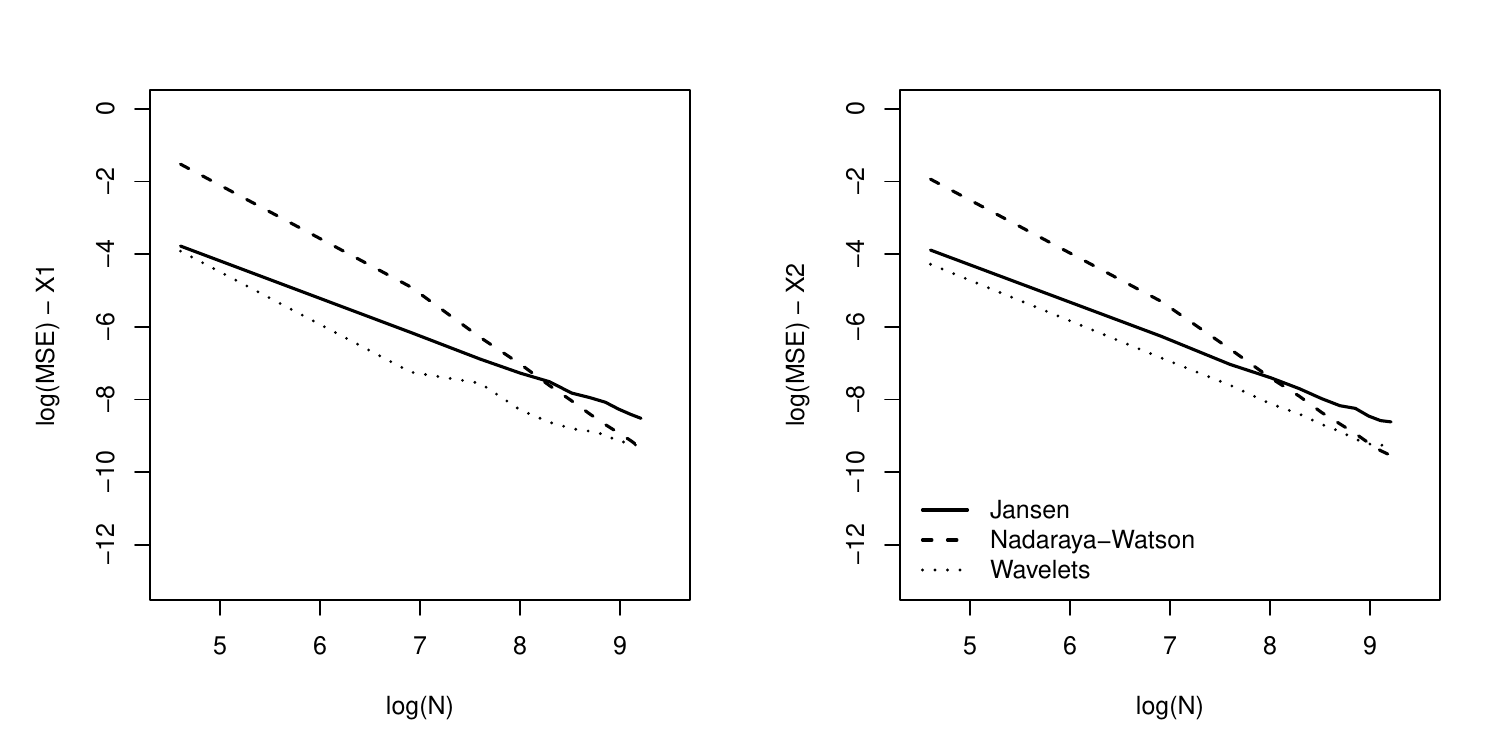}
\caption{{\small \textit{MSE for the Jansen estimators (plain), non-parametric estimator with Nadaraya-Watson (dots) and estimator with wavelets (dash-dots) in the case of the stochastic Ishigami model (Case 2) are represented as functions of $n$, which varies between $100$ and $10,000$. Left: for $X_1$. Right: for $X_2$. The graphs are plotted in logarithmic scales. }}}\label{fig:nvarie}
\end{figure}

These results suggest that the proposed non-parametric estimator
constitute an interesting alternative to the Jansen estimator, showing
less variability and potentially requiring a lower number of simulations
of the model, even in the deterministic setting of Case 1.

\section{Sobol indices for epidemiological problems}%
\label{section:epidemio}

We now consider two stochastic individual-based models of epidemiology
in continuous time. In both cases, the population is of size $N$ and
divided into compartments. Input parameters are the rates describing the
times that individuals stay in each compartment. These rates are usually
estimated from epidemiological studies or clinical trials, but there can
be uncertainty on their values due to various reasons. The restricted
size of the sample in these studies brings uncertainty on the estimates,
which are given with uncertainty intervals (classically, a 95\%
confidence interval). Different studies can provide different estimates
for the same parameters. The study populations can be subject to
selection biases. In the case of clinical trials where the efficacy of
a treatment is estimated, the estimates can be optimistic compared with
what will be the effectiveness in real-life, due to the protocol of the
trials. It is important to quantify how these uncertainties on the input
parameters can impact the results and the conclusion of an
epidemiological modelling study.

\subsection{SIR model and ODE metamodels}

In the first model, we consider the usual SIR model, with three
compartments: susceptibles, infectious and removed (e.g.
\cite{andersonbritton,bookCIMPA,diekmannheesterbeekbritton}). We denote
by $S_{t}^{N}$, $I_{t}^{N}$ and $R_{t}^{N}$ the respective sizes of the
corresponding sub-populations at time $t\geq 0$, with $S_{t}^{N}+I
_{t}^{N}+R_{t}^{N}=N$. At the population level, infections occur at the
rate $\frac{\lambda }{N} S_{t}^{N} I^{N}_{t}$ and removals at the rate
$\mu I_{t}^{N}$. The idea is that to each pair of susceptible-infectious
individuals a random independent clock with parameter $\lambda /N$ is
attached and to each infectious individual an independent clock with
parameter $\mu $ is attached.%

The input parameters are the rates $\lambda $ and $\mu $. The outpout
parameter is the final size of the epidemic, i.e. at a time $T>0$ where
$I^{N}_{T}=0$, $Y=(I^{N}_{T}+R^{N}_{T})/N$.

It is possible to describe the evolution of $(S^{N}_{t}/N,I^{N}_{t}/N,R
^{N}_{t}/N)_{t\geq 0}$ by a stochastic differential equation (SDE)
driven by Poisson point measures (see e.g. \cite{chihdr}) and it is
known that when $N\rightarrow +\infty $, this stochastic process
converges in $\mathbb{D}(\mathbb{R}_{+},\mathbb{R}^{3})$ to the unique
solution $(s_{t},i_{t},r_{t})_{t\geq 0}$ of the following system of
ordinary differential equations (e.g.
\cite{andersonbritton,bookCIMPA,diekmannheesterbeekbritton,chihdr}):
%
\begin{equation}
\label{metamodele}%
\left \{
\begin{array}{l}
\frac{ds}{dt} = - \lambda s_{t} i_{t}
\\[3pt]
\frac{di}{dt} = \lambda s_{t} i_{t} - \mu i_{t}
\\[3pt]
\frac{dr}{dt} = \mu i_{t}.
\end{array}
\right .
\end{equation}
The fluctuations associated with this convergence have also been
established. The limiting equations provide a natural deterministic
approximating meta-model (recall \cite{marrelioossdaveigaribatet})
for which sensitivity indices can be computed.

For the numerical experiment, we consider a close population of 1200
individuals, starting with $S_{0}^{1200}=1190$, $I_{0}^{1200}=10$ and
$R_{0}^{1200}=0$. The parameters distributions are uniformly distributed
with $\lambda /N \in [1/15000, 3/15000]$ and $\mu \in [1/15,3/15]$. Here
the randomness associated with the Poisson point measures is treated as
the nuisance random factor in \eqref{modelsto}.%

We compute the Jansen estimators of $S_{\lambda }$ and $S_{\mu }$ for
the deterministic meta-model \eqref{metamodele}, with $n=30{,}000$
simulations ($n(p+1)=90{,}000$ calls to the function $f$) and choose these
results as benchmark. For the estimators of $S_{\lambda }$ and
$S_{\mu }$ in the SDE, we compute the Jansen estimators with
$n=10{,}000$ (i.e. $n(p+1)=30{,}000$ calls to the function $f$), and the
estimators based on Nadaraya-Watson and on wavelet regressions with
$n=30{,}000$ simulations.

\begin{figure}
\centering
\begin{tabular}{cc}
\includegraphics[scale=0.25]{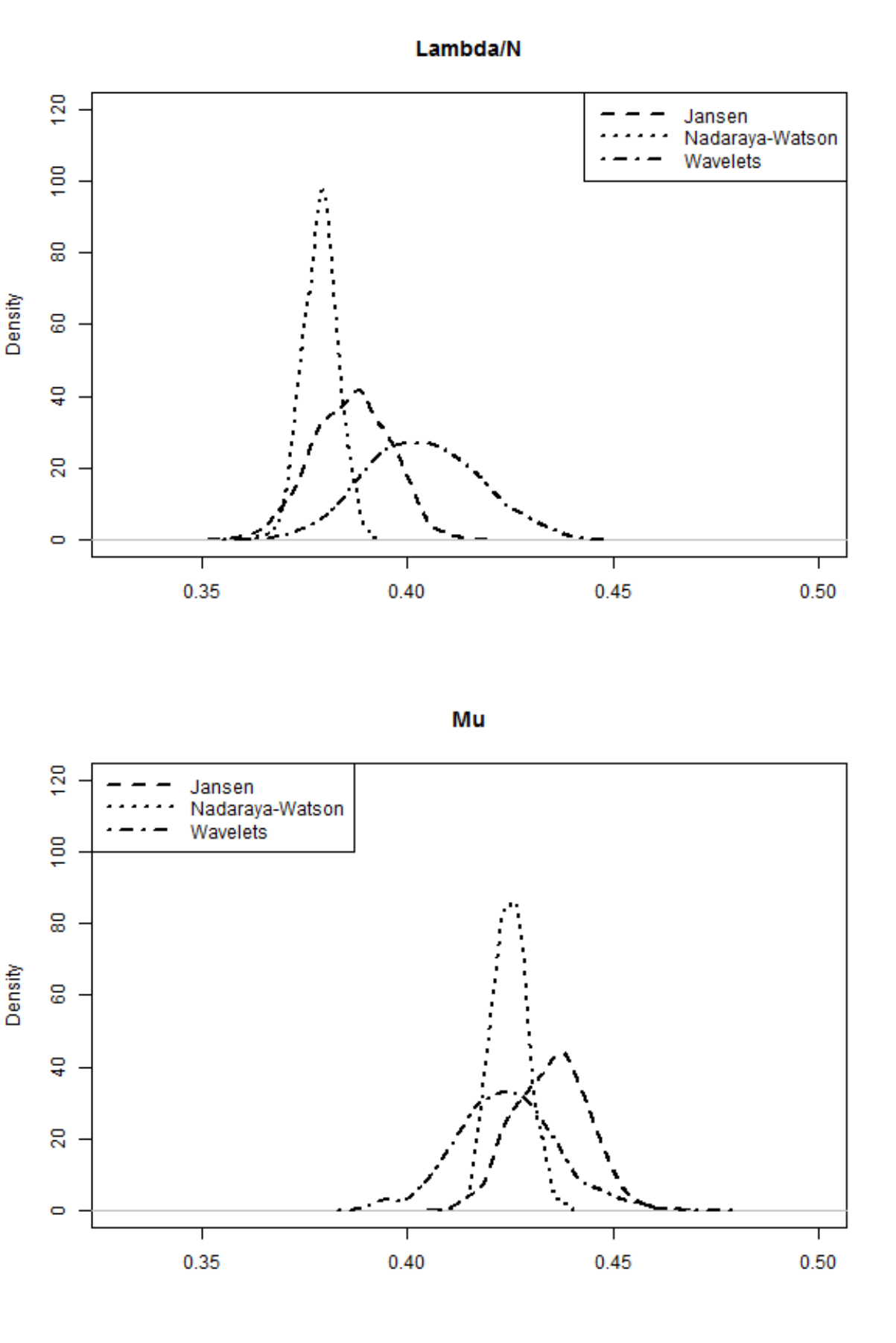} &
 \includegraphics[scale=0.25]{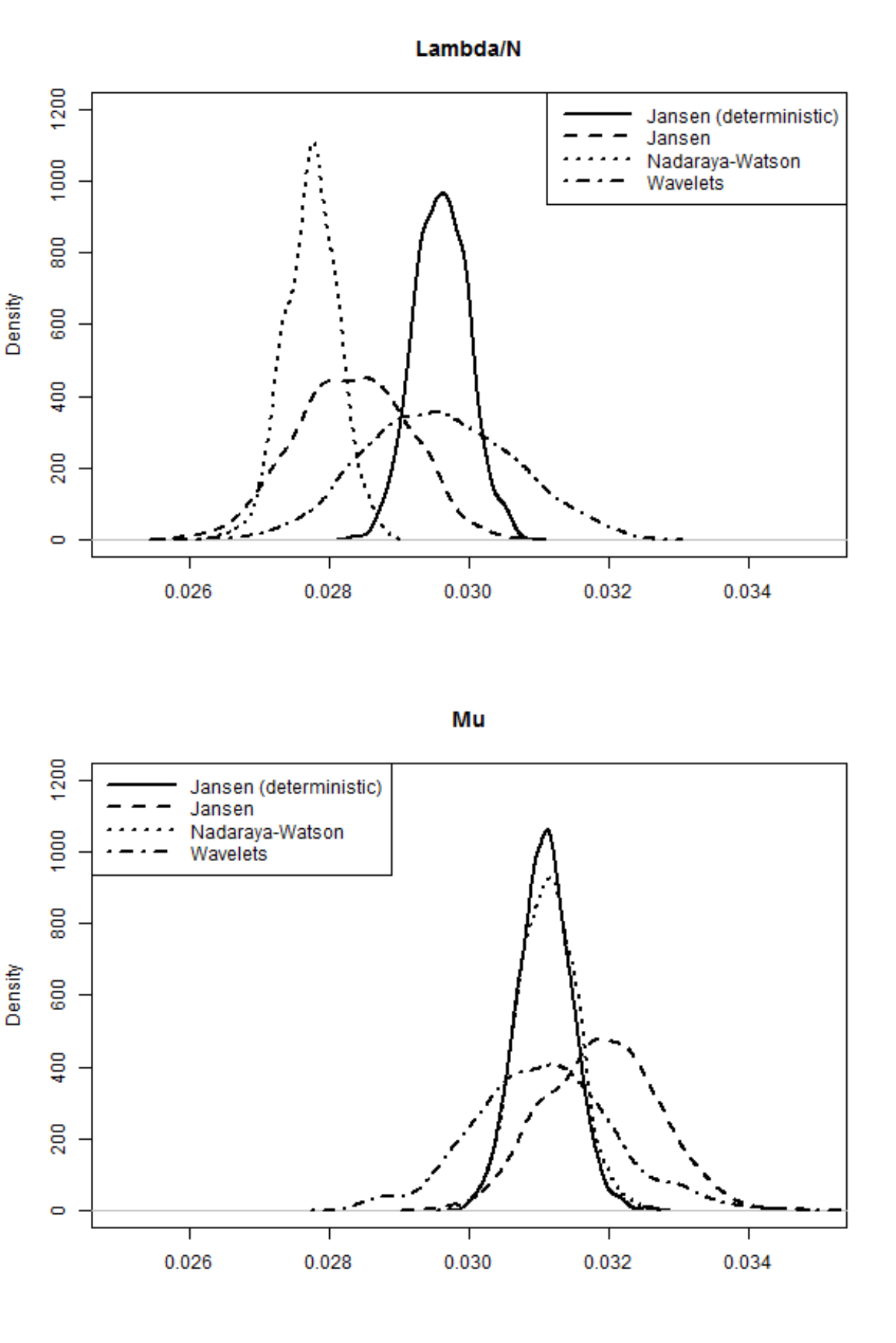}\\
(a) & (b)
\end{tabular}
\caption{{\small \textit{Estimations of the first order Sobol indices, using Jansen estimators on the meta-model with $n=10,000$ and the non-parametric estimations based on Nadaraya-Watson and wavelet regressions. (a): the distributions of the estimators of $S_{\lambda}$ and $S_{\mu}$ is approximated by Monte-carlo simulations. (b): the distributions of $\E(Y\ |\ \lambda)$ and $\E(Y\ |\ \mu)$ are approximated by Monte-Carlo simulations.}}}\label{fig:comp}
\end{figure}

Let us comment on the results. The comparison of the different
estimation methods is presented in Fig. \ref{fig:comp}. Since the
variances in the meta-model and in the stochastic model differ, we start
with comparing the distributions of $\mathbb{E}(Y\ |\ \lambda )$ and
$\mathbb{E}(Y\ |\ \mu )$ that are centered around the same value,
independently of whether the meta-model or the stochastic model is used
(Fig. \ref{fig:comp}(b)). These distributions are obtained from 1,000
Monte-carlo simulations. Because theoretical values are not available,
we take the meta-model as a benchmark. We see that the wavelet estimator
performs well for both $\lambda $ and $\mu $ while Nadaraya-Watson
regression estimator exhibit biases for $\lambda $. Jansen estimator on
the stochastic model exhibit biases for both $\lambda $ and
$\mu $.

We try to comment on the biases that are observed. When looking at Fig.
\ref{fig:exampleSIR30000}, the simulations can give very noisy $Y$'s:
extinctions of the epidemics can be seen in very short time in
simulations, due to the initial randomness of the trajectories. This
produces distributions for $Y$'s that are not unimodal or with peaks at
0, which makes the estimation of $\mathbb{E}(Y\ |\ \lambda )$ or
$\mathbb{E}(Y \ |\ \mu )$ more difficult. The wavelet estimator seems
to cope well with this situation.%

In a second time, we focus on the estimation of the Sobol indices for
the stochastic model with the SDE (we leave out the deterministic
meta-model for the reasons mentioned above). The smoothed distributions
of the estimators of $S_{\lambda }$ and $S_{\mu }$, for 1,000
Monte-Carlo replications, are presented in Fig. \ref{fig:comp}(a); the
means and standard deviations of these distributions are given in Table
\ref{table:resultSIR}. Although there is no theoretical values for
$S_{\lambda }$ and $S_{\mu }$, we can see (Table
\ref{table:resultSIR}) that the estimators of the Sobol indices with
non-parametric regressions all give similar estimates in expectation for
$\mu $. For $\lambda $, there are some discrepancies seen on Fig.
\ref{fig:comp}(a) and Table \ref{table:resultSIR}.

\begin{table}[!ht]
\caption{Estimators of the Sobol indices for $\lambda $ and $\mu $ and
their standard deviations using $n = $10{,}000 Monte-Carlo replications
of the stochastic SIR model.}
\label{table:resultSIR}
\centering
\begin{tabular}{|lccc|}
\hline
& Jansen & Nadaraya-Watson & Wavelet \\
\hline
$\widehat{S}_{\lambda }$ & 0.39 & 0.38 & 0.40 \\
s.d. & (9.2e-3) & (4.3e-3) & (1.4e-2) \\
\hline
$\widehat{S}_{\mu }$ & 0.44 & 0.42 & 0.42 \\
s.d. & (9.0e-3) & (4.4e-3) & (1.2e-2) \\
\hline
\end{tabular}
\end{table}

\begin{figure}
\centering
\includegraphics[height=8cm]{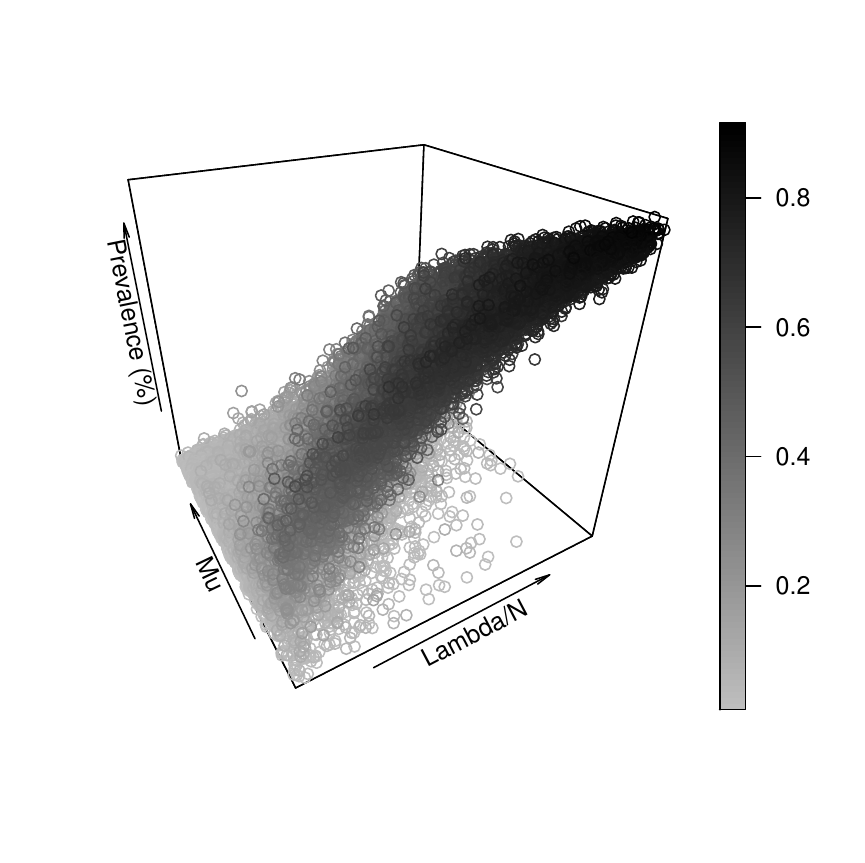}
\caption{{\small \textit{Prevalence ($Y$) simulated from the $n(p+1)=30,000$ simulations of $\lambda$ and $\mu$, for the SIR model. \label{fig:exampleSIR30000}}}}
\end{figure}

\subsection{Application to the spread of HVC among drug users}

Chronic Hepatitis C virus (HCV) is a major cause of liver failure in the
world, responsible of approximately 500,000 deaths annually
\cite{WHO}. HCV is a bloodborne disease, and the transmission remains
high in people who inject drugs (PWID) due to injecting equipment
sharing \cite{thorpe2002}. Until recently, the main approaches to
decrease HCV transmission among PWID in high income countries relied on
injection prevention and on risk reduction measures (access to sterile
equipment, opioid substitution therapies, etc.). The arrival of highly
effective antiviral treatments offers the opportunity to use the
treatment as a mean to prevent HCV transmission, by treating infected
PWID before they transmit the infection \cite{grebely2013}.\looseness=1

In this context, a stochastic, individual-based dynamic model was used
to assess the impact of the treatment on HCV transmission in PWID in
Paris area \cite{cousien1}. This model included HCV transmission on
a random graph modelling PWID social network, the cascade of care of
chronic hepatitis C and the progression of the liver disease. A brief
description of the model for HCV infection and cascade of care is
available in Fig. \ref{fig:HCVmodel}, for a detailed description and the
values and uncertainty intervals of the parameters, the reader can refer
to \cite{cousien1}. These parameters are the input of our model and
we assume for them uniform distributions on their uncertainty intervals.
Here, $Y$ is the prevalence after 10 years of simulation.

\begin{figure}
\centering
\includegraphics[scale=0.20]{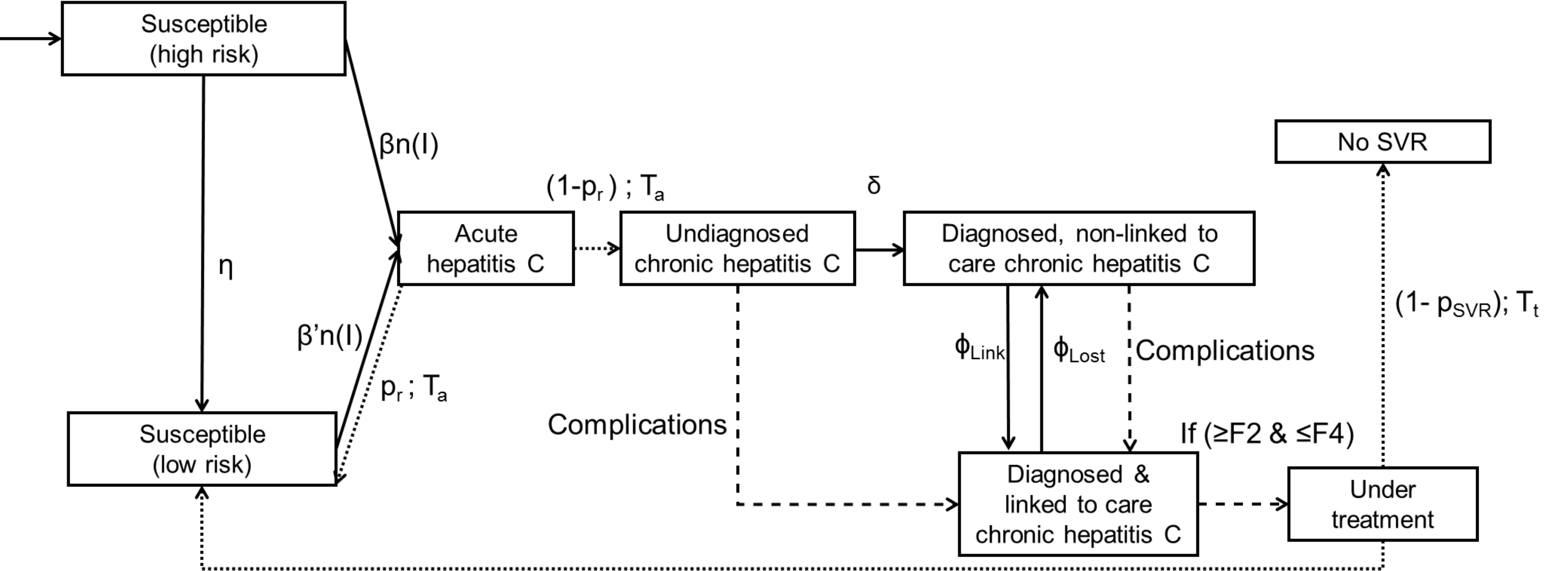}
\caption{{\small \textit{Diagram flow of infection and cascade of care modelling for HCV infection among PWID. Greek letters refer to rates, $p_r$ and $p_{SVR}$ to probabilities and $T_a$ and $T_t$ to (deterministic) time before leaving the compartment. $\beta$ depends on the status of the PWID with respect to the risk reduction measures (access to sterile injecting equipment, access to substitution therapies). $n_i$ denotes the number of infected injecting partners of the PWID. $\delta$ depends on the status of the PWID with respect to injection: active or inactive injector (i.e. before or after the cessation of injection). The liver disease progression is quantified by a score (score Metavir for the fibrosis progression) between F0 and F4 (cirrhosis). ``Complications" refers to the two cirrhosis complications: decompensated cirrhosis and hepatocellular carcinoma}}} \label{fig:HCVmodel}
\end{figure}

The parameter values used in this analysis were mainly provided by
epidemiological studies and were subject to uncertainty. This kind of
model requires high computing time, and thus the sensitivity analysis
using Monte-Carlo estimators of Sobol indices is difficult, due to the
number of simulations needed. Therefore, we focused on the seven main
ones: infection rate per partner, transition rate F0/F1 $>$ F2/F3, rate
of linkage to care and LFTU, average time to diagnosis, average time to
cessation, relative risk of infection (1st year), mortality among active
PWID. Other parameters contributions to the variance was considered as
negligible and we considered these parameters as noise in our estimates.\\

We estimate Sobol indices using the wavelet non-parametric estimator.
We used $n=10{,}000$ simulations of the model. We obtained unrealistic
results using leave-one-out cross validation procedure to select the
value of $K$ in the estimators proposed in \ref{def:estimateur2}.
However, keeping values $\widehat{\beta }^{\ell }_{jk}$ with $j<3$
produce realistic estimates. Thus, we kept all these coefficients to
produce the estimates.

For comparison, we also represented the sensitivity using a Tornado
diagram, classically used in Epidemiology. To build the Tornado diagram,
we first fix all the parameters but one to their values used in the
analysis and we let the free parameter vary in an uncertainty interval.
For each set of parameters thus obtained, the output $Y$ is computed.
Then, the parameters are sorted by decreasing variations of $Y$, and the
deviation from the main analysis results is represented in a bar plot.
We can compare the orders of the input parameters given by the Sobol
indices and by the Tornado diagram.%

\begin{figure}[!ht]
\centering
\includegraphics[scale=0.5]{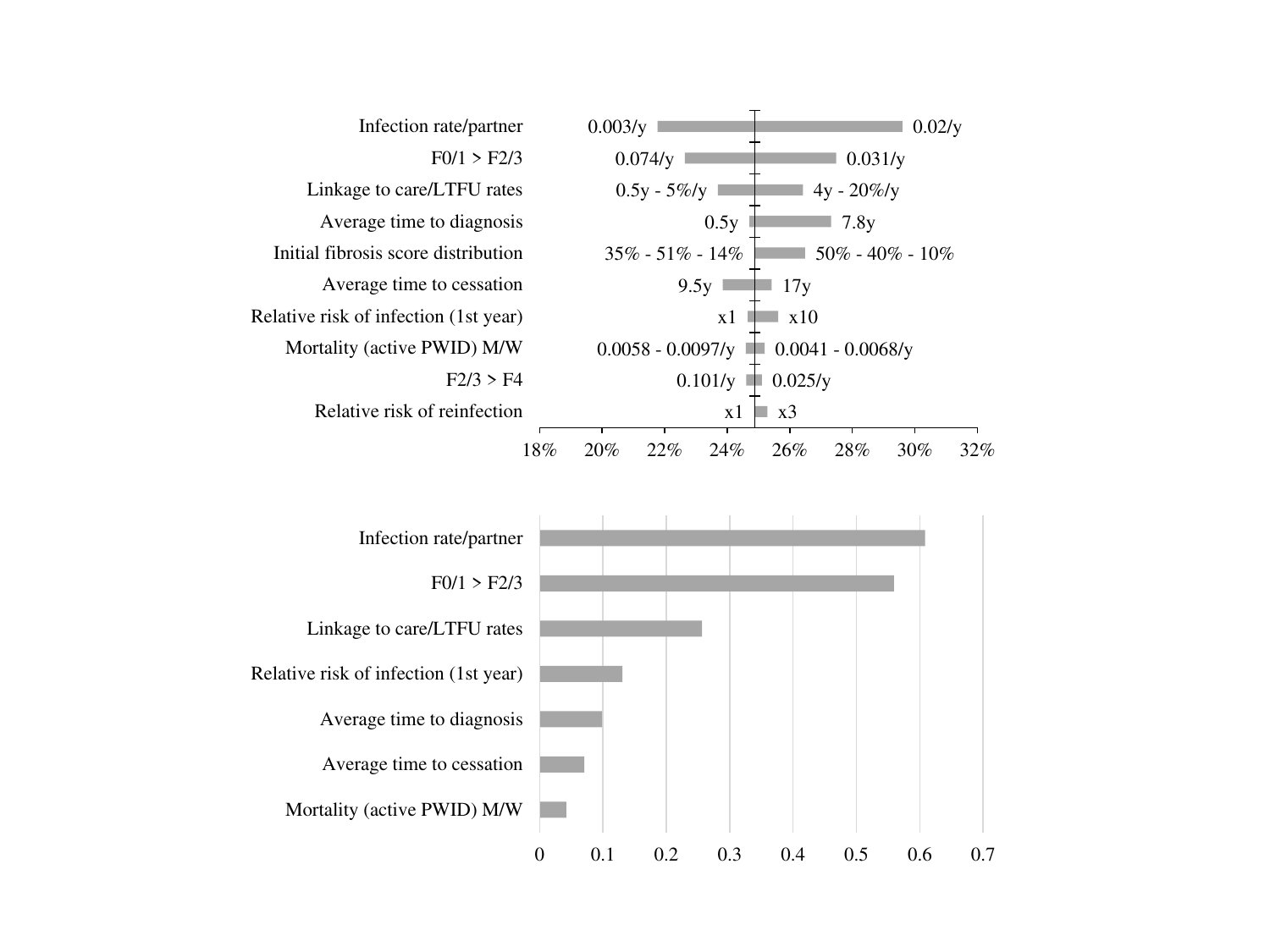}
\caption{{\small \textit{Tornado diagram (above): the variable $Y$ is plotted in abcissa and the vertical line corresponds to the expectation of $Y$ for the estimated parameters. We let each input variable vary separately between a lower bound and an upper bound, that are indicated left and right of each horizontal bar. The bars of the Tornado diagram are the corresponding values obtained for $Y$. Sobol indices (below): Sobol indices have been estimated using the wavelet estimators. Parameters have been sorted by decreasing values of their Sobol indices. What can be compared is the order of the various input variables in each method. LTFU=loss to follow-up, HCC=Hepatocellular carcinoma, M=Male, F=Female. ``Cessation" refers to the cessation of the injections. ``F0/F1 $>$ F2/F3" refers to the transition rate from a fibrosis score F0 or F1 to a fibrosis score F2 or F3 (and similarly for other rates). }}}\label{fig:HCVsa}
\end{figure}

The results are presented in Figure \ref{fig:HCVsa}. Since the Sobol
indices can be interpreted as the contribution of each parameter to the
variance of $Y$, we can thus see that a large part of the variance of
$Y$ is explained by the infection rate per infected partner alone, with
a Sobol index of 0.6, and by the transition rate from a fibrosis score
of F0/F1 to a score of F2/F3, with a Sobol index of 0.55. Next comes the
linkage to care/loss to follow-up rate. The rankings of the input
parameters obtained by the Sobol indices and the Tornado diagram
(obtained in \cite{cousien1}) are in accordance for the main
parameters. For the Tornado diagram, the most sensitive parameters (the
infection rate per infected injecting partner, the transition rate from
a fibrosis score of F0/F1 to a score of F2/F3 and the combination of the
linkage to care/loss to follow-up rate) were also varied together to
estimate the impact of the uncertainty about the linkage to care of
PWID. The Tornado diagram, which explores a much smaller region of the
parameter space by the way it is constructed, detects more noisy
contributions for the other factors. This appears, in the Tornado, in
the group of parameters having similar Sobol indices (average time to
diagnosis and cessation, relative risk of infection, mortality,
F2/F3$>$F4).

\section{Conclusion}%
\label{section:conclusions}

Sensitivity analysis is a key step in modelling studies, in particular
in epidemiology. Models often have a high number of parameters, which
are often seen as degrees of freedom to test scenarii and take into
account several interplaying phenomena and factors. The computation
of Sobol indices can indicate, among a long list of input parameters,
which ones can have an important impact on the outputs. The classical
estimators, like the Jansen estimator, require a large amount of
requests to the function $f$ that generates the output from the inputs.
The reason is that the Sobol indices are approximated, in these cases,
by quantities involving imbricated sums where parameters vary one by one.%

The literature on sensitivity analysis focuses on outputs that depend
deterministically on the inputs. When there is randomness, it is natural
to propose new approximations based on non-parametric estimations that
require a lower number of calls to $f$ since information brought by
simulations with close input parameters can also be used. No meta-model
is requested. Numerical study on toy models show that these estimators
can be used in deterministic settings too.%

Independently and at the same time as us \cite{cousienPhD},
Sol\'{i}s \cite{solisthesis,solis} introduced an estimator of the
Sobol indices of order 1 based on Nadaraya-Watson regressions. We hence
focus in this paper on an estimator of the Sobol indices based on
wavelet decompositions. For both of them, an elbow effect is proved:
under sufficient regularities, convergence rates of order $1/
\sqrt{n}$ can be achieved. On numerical toy examples, we obtained a
better MSE with the wavelet estimators than with the Jansen estimator
of same complexity. The non-parametric estimators allow a better
exploration of the parameter space: for each simulation, the whole set
of input parameters is drawn afresh. Compared with the Nadaraya-Watson
estimator, the wavelet estimator is adaptative, which means that the
unknown regularity of the model underlying the data does not need to be
known to calibrate the estimator. On simulations, our estimator behaves
similarly with Nadaraya-Watson estimator. When well-calibrated they can
overcome some smoothing biases that can appear when the output is very
noisy, which is the case in epidemic scenarii where there can be either
large outbreaks or quick extinction due to stochasticity, for example.%

Notice also that our proofs in the present paper are much shorter than
the proofs needed to study the estimator based on Nadaraya-Watson
regression. First, the wavelet estimator is a projection estimator and
the difficulties related with the fact that there is a fraction in the
Nadaraya-Watson estimator disappear. Second, we use elegant techniques
(developed independently from sensitivity analysis) on empirical
processes and concentration inequalities due to Castellan
\cite{castellan} to adapt the results of Laurent and Massard in the
Gaussian case \cite{laurentmassart}.%

This first order index $S_{\ell }$ corresponds to the sensitivity of the
model to $X_{\ell }$ alone. Higher order indices can also be defined
using ANOVA decomposition: considering $(\ell ,\ell ') \in \{1,\dots
, p\}$, we can define the second order sensitivity, corresponding to the
sensitivity of the model to the interaction between $X_{\ell }$ and
$X_{\ell '}$ index by
%
\begin{equation}
S_{\ell \ell '}=\frac{\mbox{Var}\big (\mathbb{E}[Y\ |\ X_{\ell },X_{
\ell '}]\big )}{\mbox{Var}(Y)} - S_{\ell } - S_{\ell '}
\label{def:Sobol2}
\end{equation}
We can also define the total sensitivity indices by
%
\begin{equation}
S_{T_{\ell }} = \sum _{L \subset \{1,\dots , p\} \, |\, \ell \in L} S
_{L}.
\label{def:SobolTot}
\end{equation}
These indices allow to assess 1) the sensitivity of the model to each
parameter taken separately and 2) the possible interactions, which are
quantified by the difference between the total order and the first order
index for each parameter. Estimation of higher order indices using
non-parametric techniques would be an interesting subject for further
researches.

\section{Proofs}%
\label{section:proof_th}

\subsection{Proof of Theorem \ref{th1}}

We follow the scheme of the proof of Theorem 1 in
\cite{laurentmassart}. The main difficulty here is that we are not in
a Gaussian framework and that we use the empirical process $\bar{
\gamma }_{n}$, which introduces much technical difficulties.

In the sequel, $C$ denotes a constant that can vary from line to line.%

Using Lemma \ref{lemme1}, we concentrate on the MSE $\mathbb{E}\big ((
\widehat{V}_{\ell }-V_{\ell })^{2}\big )$. First, we will prove that:
%
\begin{equation}
\mathbb{E}\Big [\Big (\widehat{V}_{\ell }-V_{\ell }- \zeta _{n} \Big )^{2}
\Big ]
\leq \inf _{\mathcal{J}\subset \{-1,\dots ,J_{n}\}} \mathbb{E}\Big [
\Big (-\widehat{V}_{\mathcal{J},\ell }+\mbox{pen}(\mathcal{J})+V_{
\ell }+ \zeta _{n} \Big )^{2}_{+}\Big ]+
\frac{C \log _{2}^{2}(n)}{n^{3/2}},
\label{oracle:etape1}
\end{equation}
where $\widehat{V}_{\mathcal{J},\ell }$ has been defined in
\eqref{def:est_interm2}. The penalization term associated to a subset
$\mathcal{J}\subset \{-1,\dots J_{n}\}$ has been defined in
\eqref{def:penalty}. Then, considering the first term in the r.h.s. of
\eqref{oracle:etape1}, we prove:
%
\begin{align}
\mathbb{E}\Big [\Big (-\widehat{V}_{\mathcal{J},\ell }+\mbox{pen}(
\mathcal{J})+V_{\ell }+ \zeta _{n} \Big )^{2}_{+}\Big ] \leq C \Big (\|h
_{\ell }- h_{\mathcal{J},\ell }\|_{2}^{4} + \frac{\mbox{Card}^{2}(
\mathcal{J})}{n^{2}}\Big )
\label{oracle:etape2}
\end{align}
\textbf{Step 1:}%

\noindent From \eqref{lien_pen}, and letting $A_{\mathcal{J}}=\widehat{V}_{
\mathcal{J},\ell }-\mbox{pen}(\mathcal{J})-V_{\ell }- \zeta _{n}$, we
have:
\begin{equation*}
\widehat{V}_{\ell }-V_{\ell }- \zeta _{n}=
\sup _{\mathcal{J}\subset \{-1,\dots ,J_{n}\}} A_{\mathcal{J}}.
\end{equation*}
Since
\begin{equation*}
\sup _{\mathcal{J}} A_{\mathcal{J}}= \sup _{\mathcal{J}} \big (A_{
\mathcal{J}}\big )_{+}\ \ind _{\{\sup _{\mathcal{J}} A_{\mathcal{J}}
\geq 0\}}- \inf _{\mathcal{J}} \big (A_{\mathcal{J}}\big )_{-} \ \ind
_{\{\sup _{\mathcal{J}} A_{\mathcal{J}}< 0\}},
\end{equation*}
we obtain by taking the absolute values that
\begin{equation*}
\Big |\sup _{\mathcal{J}} A_{\mathcal{J}}\Big | \leq \max \Big [
\sup _{\mathcal{J}} \big (A_{\mathcal{J}}\big )_{+} , \inf _{\mathcal{J}}
\big (A_{\mathcal{J}}\big )_{-} \Big ].
\end{equation*}
This provides that
%
\begin{align}
\mathbb{E}\Big (\sup _{\mathcal{J}} A^{2}_{\mathcal{J}}\Big )\leq
&
\sum _{\mathcal{J}\subset \{-1,\dots ,J_{n}\}} \mathbb{E}\Big ( \big (A
_{\mathcal{J}}\big )^{2}_{+}\Big )+
\inf _{\mathcal{J}\subset \{-1,\dots ,J_{n}\}} \mathbb{E}\Big (\big (A
_{\mathcal{J}}\big )^{2}_{-}\Big )
\nonumber
\\
\leq
& \sum _{\mathcal{J}\subset \{-1,\dots ,J_{n}\}} \mathbb{E}
\Big ( \big (A_{\mathcal{J}}\big )^{2}_{+}\Big )
+\inf _{\mathcal{J}\subset \{-1,\dots ,J_{n}\}} \mathbb{E}
\Big ( \big (- \widehat{V}_{\mathcal{J},\ell } + \mbox{pen}(\mathcal{J})+
V_{\ell }+\zeta _{n}\big )^{2}_{+} \Big ).
\label{etape:but1}
\end{align}
The second term corresponds to what appears in \eqref{oracle:etape1} and
will be treated in Step 4 to obtain \eqref{oracle:etape2}. Let us
consider the first term of the r.h.s.

From \eqref{def:penalty}, we have:
%
\begin{align}
\mbox{pen}(\mathcal{J})=
& \sum _{j\in \mathcal{J}} w(j)= \frac{K}{n}
\sum _{j\in \mathcal{J}} (2^{j}+\log 2)
= \mbox{pen}_{1}(\mathcal{J})+
\mbox{pen}_{2}(\mathcal{J}),
\label{def:pen}
\end{align}
with
%
\begin{align}
\mbox{pen}_{1}(\mathcal{J})=
& \frac{K }{n}\sum _{j\in \mathcal{J}} 2^{j}
\label{def:penalisations}%
\\
\mbox{pen}_{2}(\mathcal{J})=
& \mbox{pen}(\mathcal{J})-\mbox{pen}
_{1}(\mathcal{J})= \frac{K}{n} \mbox{Card}(\mathcal{J})\ \log 2 .
\end{align}
Using this, we start by rewriting
\begin{align}
A_{\mathcal{J}}=
& \widehat{V}_{\mathcal{J},\ell } -\mbox{pen}(
\mathcal{J})-V_{\ell }- \zeta _{n}
\nonumber
\\
=
& \| \widehat{h}_{\mathcal{J},\ell }\|^{2}_{2}-\mbox{pen}(
\mathcal{J}) - \|h_{\ell }\|^{2}_{2} -\zeta _{n}
\nonumber
\\
=
& \big (\| \widehat{h}_{\mathcal{J},\ell }-h_{\mathcal{J},\ell } \|
^{2}_{2} + \|h_{\mathcal{J},\ell }\|^{2}_{2} + 2\langle \widehat{h}
_{\mathcal{J},\ell }-h_{\mathcal{J},\ell }, h_{\mathcal{J},\ell }
\rangle \big )
\nonumber
\\
& - \big (\|h_{\ell }-h_{\mathcal{J},\ell }\|^{2}_{2} +\|h_{
\mathcal{J},\ell }\|^{2}_{2}+2\langle h_{\ell }-h_{\mathcal{J},\ell },
h_{\mathcal{J},\ell } \rangle \big )-\zeta _{n}-\mbox{pen}(\mathcal{J})
\nonumber
\\
=
& \| \widehat{h}_{\mathcal{J},\ell }-h_{\mathcal{J},\ell } \|^{2}
_{2} -\mbox{pen}_{1}(\mathcal{J})
 + 2\langle \widehat{h}_{\mathcal{J},\ell }-h_{\mathcal{J},\ell }, h
_{\mathcal{J},\ell } \rangle -\|h_{\ell }-h_{\mathcal{J},\ell }\|^{2}
_{2} -\zeta _{n} - \mbox{pen}_{2}(\mathcal{J}),
\label{etape3}
\end{align}
since $\langle h_{\ell }-h_{\mathcal{J},\ell }, h_{\mathcal{J},\ell }
\rangle =0$ by definition of $h_{\mathcal{J},\ell }$ as projection of
$h_{\ell }$ on the subspace generated by $\{\psi _{jk},\ j\in
\mathcal{J},\, k\in \mathbb{Z}\}$.%

Thus:
%
\begin{multline}
\mathbb{E}\Big ( \big (A_{\mathcal{J}}\big )^{2}_{+}\Big )\leq 2
\mathbb{E}\Big ( \big (\| \widehat{h}_{\mathcal{J},\ell }-h_{
\mathcal{J},\ell } \|^{2}_{2} - \mbox{pen}_{1}(\mathcal{J})\big )^{2}
\Big )
\\
+
2 \mathbb{E}\Big ( \big (2\langle \widehat{h}_{\mathcal{J},\ell }-h
_{\mathcal{J},\ell }, h_{\mathcal{J},\ell } \rangle -\|h_{\ell }-h
_{\mathcal{J},\ell }\|^{2}_{2} -\zeta _{n} -\mbox{pen}_{2}(\mathcal{J})
\big )^{2}\Big ).
\label{etape8}
\end{multline}
The first term in the r.h.s. is treated in Step 2, and the second term
in Step 3. After summation over $\mathcal{J}\subset \{-1,\dots , J
_{n}\}$, this provides an upper bound for the first term in the r.h.s.
of \eqref{etape:but1} which provides \eqref{oracle:etape1}.%

\medskip\noindent
\textbf{Step 2: Upper bound of the first term in the r.h.s. of \eqref{etape8}}%

\medskip\noindent
\underline{Reformulation of $\| \widehat{h}_{\mathcal{J},\ell }-h_{\mathcal{J},\ell } \|^{2}_{2}$}

\medskip
The first term in the r.h.s. of \eqref{etape3} is the approximation
error of $h_{\mathcal{J}}$ by $\widehat{h}_{\mathcal{J},\ell }$ and
equals
\begin{equation*}
\| \widehat{h}_{\mathcal{J},\ell }-h_{\mathcal{J},\ell } \|^{2}_{2}=
\sum _{j\in \mathcal{J}} \sum _{k\in \mathbb{Z}} \big (\widehat{\beta }
_{jk}-\beta _{jk}\big )^{2}=\sum _{j\in \mathcal{J}} \sum _{k\in
\mathbb{Z}} \bar{\gamma }_{n}\big (\psi _{jk}\big )^{2}.
\end{equation*}
To control it, let us introduce, for coefficients $a=(a_{jk},\ -1
\leq j\leq J_{n},\ k\in \mathbb{Z})$, the set
\begin{equation*}
\mathcal{F}_{1,\mathcal{J}}=\Big \{f=\sum _{j\in \mathcal{J}}
\sum _{k\in \mathbb{Z}} a_{jk} \psi _{jk},\ a_{jk}\in \mathbb{Q},\ \|a
\|_{2}\leq 1 \Big \},
\end{equation*}
which is countable and dense in the unit ball of $L^{2}([0,1])$. Thus,
%
\begin{align}
\Big (\sum _{j\in \mathcal{J}} \sum _{k\in \mathbb{Z}} \bar{\gamma }_{n}
\big (\psi _{jk}\big )^{2}\Big )^{1/2}=
& \sup _{\|a\|_{2}\leq 1} \Big |
\sum _{j\in \mathcal{J}} \sum _{k\in \mathbb{Z}} a_{jk} \bar{\gamma }
_{n}\big (\psi _{jk}\big )\Big |
\nonumber
\\
=
& \sup _{\|a\|_{2}\leq 1} \Big |\bar{\gamma }_{n}\Big (
\sum _{j\in \mathcal{J}} \sum _{k\in \mathbb{Z}} a_{jk} \psi _{jk}\Big )
\Big |
\nonumber
\\
=
& \sup _{f \in \mathcal{F}_{1,\mathcal{J}}} \big |\bar{\gamma }_{n}(f)
\big |:= \chi _{n}(\mathcal{J}).
\label{etape7}
\end{align}
Let us introduce, for $\rho >0$,
\vspace*{-3pt}
\begin{equation}
\label{def:Omega}
\Omega _{\mathcal{J}}(\rho )=\big \{\forall j\in \mathcal{J},\,
\sum _{k\in \mathbb{Z}} \big |\bar{\gamma }_{n}(\psi _{jk})\big |\leq
\rho 2^{-j/2}\big \}.
\end{equation}
Then, to upper bound the first term in \eqref{etape8}, we can write:
\vspace*{-3pt}
\begin{align}
\mathbb{E}\Big ( \big (\| \widehat{h}_{\mathcal{J},\ell }-h_{
\mathcal{J},\ell } \|^{2}_{2} - \mbox{pen}_{1}(\mathcal{J})\big )^{2}
\Big )\leq 2A_{1} (\mathcal{J}) + 2A_{2}(\mathcal{J})
\label{decompo:A1A2}
\end{align}
where, for $\chi _{n}(\mathcal{J})$ defined in \eqref{etape7},
%
\begin{align}
& A_{1}(\mathcal{J})= \mathbb{E}\Big ( \big (\chi _{n}^{2}(\mathcal{J}) \ind
_{\Omega _{\mathcal{J}}(\rho )} - \mbox{pen}_{1}(\mathcal{J})\big )^{2}
\Big ),
\nonumber
\\
\mbox{and }\quad
& A_{2}(\mathcal{J})= \mathbb{E}\Big ( \chi _{n}^{4}(
\mathcal{J}) \ind _{\Omega ^{c}_{\mathcal{J}}(\rho )}\Big ).
\end{align}
The upper bounds of $A_{1}(\mathcal{J})$ and $A_{2}(\mathcal{J})$ make
the object of the remainder of Step 2. We use ideas developed in
\cite{castellan}.%

\medskip\noindent
\underline{Upper bound for $A_{1}(\mathcal{J})$}%

\medskip
To upper bound $A_{1}(\mathcal{J})$, we use the identity
%
\begin{equation}
A_{1}(\mathcal{J})=\int _{0}^{+\infty } 2t \ \mathbb{P}\big (\chi _{n}
^{2}(\mathcal{J}) \ind _{\Omega _{\mathcal{J}}(\rho )} - \mbox{pen}
_{1}(\mathcal{J})>t\big )\ dt,
\label{debut:A1}%
\end{equation}
and look for deviation inequalities of $\chi _{n}^{2}(\mathcal{J}) \ind
_{\Omega _{\mathcal{J}}(\rho )}$. Then, estimates of the probability of
$\Omega ^{c}_{\mathcal{J}}(\rho )$ are studied to control $A_{2}(
\mathcal{J})$.

Recall that $\chi _{n}(\mathcal{J})$ (resp. $\Omega _{\mathcal{J}}(
\rho )$) has been defined in \eqref{etape7} (resp. \eqref{def:Omega}).
The supremum in \eqref{etape7} is obtained for
\vspace*{-3pt}
\begin{equation}
\bar{a}_{jk} = \frac{\bar{\gamma }_{n}(\psi _{jk})}{\chi _{n}(
\mathcal{J})}.
\end{equation}
On the set $\Omega _{\mathcal{J}}(\rho )\cap \{\chi _{n}(\mathcal{J}) >
z\}$, for a constant $z>0$ that shall be fixed in the sequel, we have
for all $j\in \mathcal{J}$,
\vspace*{-3pt}
\begin{equation*}
\sum _{k\in \mathbb{Z}} \big |\bar{a}_{jk} \big |= \frac{
\sum _{k\in \mathbb{Z}} \big |\bar{\gamma }_{n}(\psi _{jk})\big |}{\chi
_{n}(\mathcal{J})}\leq \frac{\rho 2^{-j/2}}{z}.
\end{equation*}
As a consequence, on the set $\Omega _{\mathcal{J}}(\rho )\cap \{\chi
_{n}(\mathcal{J} )> z\}$, we can restrict the research of the optima to
the set
\vspace*{-3pt}
\begin{multline}
\Lambda _{\mathcal{J}}=\Big \{f=\sum _{j\in \mathcal{J}}
\sum _{k\in \mathbb{Z}}a_{jk}\psi _{jk}\in \mathcal{F}_{1,\mathcal{J}},
\\
\mbox{ and } \ a_{jk}=0 \mbox{ if }j\notin \mathcal{J},\ \sum _{k\in \mathbb{Z}}\big |a_{jk}\big |
\leq \frac{\rho 2^{-j/2}}{z}\mbox{ if }j\in \mathcal{J}\Big \},
\end{multline}
which is countable.\\

We can then use Talagrand inequality (see
\cite[p.170]{massart_concentration}) to obtain that for all
$\eta >0$ and $x>0$,
%
\begin{equation}
\mathbb{P}\Big (\sup _{f \in \Lambda _{\mathcal{J}}} \big |\bar{\gamma }
_{n}(f)\big |\geq (1+\eta )\mathbb{E}\big (
\sup _{f \in \Lambda _{\mathcal{J}}} \big |\bar{\gamma }_{n}(f)\big |
\big )+\sqrt{2 \nu _{n} x}+\big (\frac{1}{3}+\frac{1}{\eta }\big )b_{n}
x\Big )\leq e^{-x},
\label{talagrand}
\end{equation}
where the quantities $\nu _{n}$ and $b_{n}$ can be chosen respectively
as $\nu _{n}=M^{2}/n$ and $b_{n}=2M \|\psi \|_{\infty }\rho
\mbox{Card}(\mathcal{J}) /nz$.%

Indeed, $\nu _{n}$ is an upper bound of:
%
\begin{equation}
\frac{1}{n} \sup _{f \in \Lambda _{\mathcal{J}}} \mbox{Var}\Big (Y_{1} f
\big (G_{\ell }(X_{\ell }^{1})\big )\Big )\leq \frac{M^{2}}{n}
\sup _{f \in \Lambda _{\mathcal{J}}}\|f\|_{2}^{2}\leq \frac{M^{2}}{n},
\end{equation}
where the last inequality comes from the definition of $
\Lambda _{\mathcal{J}}$ and $\mathcal{F}_{1,\mathcal{J}}$.

As for the term $b_{n}$, it is an upper bound of:
%
\begin{multline}
\frac{1}{n} \sup _{f \in \Lambda _{\mathcal{J}}}\sup _{(u,y)\in [0,1]\times [-M,M]} \Big | y f(u)-\mathbb{E}\Big (Y_{1} f\big (G_{\ell }(X_{\ell }^{1})\big )\Big )\Big |
\\
\leq
 \frac{2 M}{n} \sum _{j\in \mathcal{J}} \sum _{k\in \mathbb{Z}} |a
_{jk}| 2^{j/2} \|\psi \|_{\infty }
\leq
 \frac{2 M \|\psi \|_{\infty }}{n} \sum _{j\in \mathcal{J}} \frac{
\rho 2^{-j/2}}{z} 2^{j/2}= \frac{2 M \|\psi \|_{\infty }\rho \ \mbox{Card}(\mathcal{J})}{n \ z} ,
\end{multline}
if $f=\sum _{j\in \mathcal{J}}\sum _{k\in \mathbb{Z}} a_{jk}\psi _{jk}$.
For the expectation appearing in the probability in the r.h.s. of
\eqref{talagrand}, we have:
%
\begin{align}
& \mathbb{E}\big (\sup _{f \in \Lambda _{\mathcal{J}}} \big |\bar{\gamma
}_{n}(f)\big |\big )\leq \mathbb{E}\big (\chi _{n}(\mathcal{J})\big )
\leq \sqrt{\mathbb{E}\big (\chi _{n}^{2}(\mathcal{J})\big )}= \sqrt{
\sum _{j\in \mathcal{J}}\sum _{k\in \mathbb{Z}} \mathbb{E}\big (\bar{
\gamma }_{n}^{2}(\psi _{jk})\big )}
\nonumber
\\
&
\hspace{2.8cm}
= \sqrt{\sum _{j\in \mathcal{J}}\sum _{k\in \mathbb{Z}} \frac{1}{n}
\mbox{Var}\big (Y_{1} \psi _{jk}(G_{\ell }(X^{1}_{\ell }))\big )}
\nonumber
\\
&
\hspace{2.8cm}
\leq \sqrt{\sum _{j\in \mathcal{J}}\sum _{k\in \mathbb{Z}}
\frac{1}{n}\mathbb{E}\big (Y_{1}^{2} \psi ^{2}_{jk}(G_{\ell }(X^{1}_{
\ell }))\big )}
\nonumber
\\
&
\hspace{2.8cm}
\leq \frac{M}{\sqrt{n}}\sqrt{\sum _{j\in \mathcal{J}}
\sum _{k\in \mathbb{Z}} \int _{0}^{1} \psi ^{2}_{jk}(u)du} \leq \frac{M}{
\sqrt{n}}\sqrt{ C' \sum _{j\in \mathcal{J}}2^{j}}
\label{etape14}
\end{align}
by the fact that the wavelets $\psi _{jk}$ have compact supports and
satisfy $\|\psi _{jk}\|_{2}^{2}=1$. The constant $C'$ in
\eqref{etape14} is the number of wavelets $(\psi _{0k})_{k\in
\mathbb{Z}}$ that intersect $[0,1]$. Thus for a given $j\geq 0$, the
number of wavelets $(\psi _{jk})_{k\in \mathbb{Z}}$ that instersect
$[0,1]$ is of order $2^{j} C'$.

Because we have on $\Omega _{\mathcal{J}}(\rho )\cap \{\chi _{n}(
\mathcal{J}) > z\}$ that $\chi _{n}(\mathcal{J})=
\sup _{f \in A_{\mathcal{J}}} \big |\bar{\gamma }_{n}(f)\big |$, we deduce
that $\sup _{f \in A_{\mathcal{J}}} \big |\bar{\gamma }_{n}(f)\big |
\geq \chi _{n}(\mathcal{J}) \ind _{\Omega _{\mathcal{J}}(\rho )\cap \{
\chi _{n}(\mathcal{J} > z\}}$. Then, Equations
\eqref{talagrand}-\eqref{etape14} become:
\vspace*{-9pt}
\begin{multline*}
\mathbb{P}\Big (\chi _{n}(\mathcal{J}) \ind _{\Omega _{\mathcal{J}}(
\rho )\cap \{\chi _{n}(\mathcal{J}) > z\}} \geq (1+\eta )M
\sqrt{\frac{C' \sum _{j\in \mathcal{J}} 2^{j}}{n}}+\sqrt{\frac{2 M
^{2} x}{n}}
\\
+\big (\frac{1}{3}+\frac{1}{\eta }\big )\frac{2 M \|\psi \|_{\infty }
\rho \ \mbox{Card}(\mathcal{J})}{n \ z} x\Big )\leq e^{-x}.
\end{multline*}
Choosing $z=\sqrt{\frac{2x}{n}}\big (\frac{1}{3}+\frac{1}{\eta }
\big ) \|\psi \|_{\infty }$, we obtain:
\begin{equation*}
\mathbb{P}\Big (\chi _{n}(\mathcal{J}) \ind _{\Omega _{\mathcal{J}}(
\rho )\cap \{\chi _{n}(\mathcal{J}) > z\}} \geq (1+\eta )M
\sqrt{\frac{C' \sum _{j\in \mathcal{J}} 2^{j}}{n}}
+(1+\rho \mbox{Card}(\mathcal{J}))M \sqrt{\frac{2 x}{n}}\Big )
\leq e^{-x}.
\end{equation*}
For the choice of $\rho =\big (\frac{1}{3}+\frac{1}{\eta }\big ) \|
\psi \|_{\infty }$, the r.h.s. in the probability above is larger than
$z=\sqrt{\frac{2x}{n}}\big (\frac{1}{3}+\frac{1}{\eta }\big ) \|
\psi \|_{\infty }$, and we can get rid of the constraint $\{\chi _{n}(
\mathcal{J})>z\}$. Finally, choosing $x=x_{\mathcal{J}}+\xi $, with
%
\begin{equation}
\label{choix:xJ-1}
x_{\mathcal{J}}= \log \Big (\sum _{j\in \mathcal{J}} 2^{j}\Big ),
\end{equation}
we obtain by using $(a+b)^{2}\leq 2a^{2}+2b^{2}$ that:
\begin{multline*}
\mathbb{P}\Big (\chi ^{2}_{n}(\mathcal{J}) \ind _{\Omega _{\mathcal{J}}(
\rho )} - \frac{2}{n} \Big [(1+\eta )^{2} M^{2} C' \sum _{j\in
\mathcal{J}} 2^{j} +2 (1+\rho \mbox{Card}(\mathcal{J}))^{2} M^{2} x
_{\mathcal{J}} \Big ]
\geq h_{\mathcal{J}}(\xi )\Big )
\\
\leq e^{-x_{\mathcal{J}}}e^{-\xi },
\end{multline*}
where
%
\begin{equation}
h_{\mathcal{J}}(\xi )=\frac{4(1+\rho \mbox{Card}(\mathcal{J}))^{2} M
^{2}}{n} \xi .
\label{def:hxi}
\end{equation}
The square bracket in the l.h.s. inside the probability can be upper
bounded by $n\mbox{pen}_{1}(\mathcal{J})= K \sum _{j\in \mathcal{J}}
2^{j}$, for an appropriate constant $K$ that does not depend on
$\mathcal{J}$. Indeed, denoting by $J_{\mbox{{\scriptsize max}}}=
\max \mathcal{J}$, we have that $\mbox{Card}^{2}(\mathcal{J})\leq J
_{\mbox{{\scriptsize max}}}^{2}$ while $\sum _{j\in \mathcal{J}}2^{j}
\geq 2^{J_{\mbox{{\scriptsize max}}}}$. Since $2^{j}\geq j^{2}$ for all
interger $j\not =3$, the result follows. Then:
%
\begin{align}
\mathbb{P}\Big (\chi ^{2}_{n}(\mathcal{J}) \ind _{\Omega _{\mathcal{J}}(
\rho )} - \mbox{pen}_{1}\big (\mathcal{J}\big )
\geq h_{\mathcal{J}}(
\xi )\Big )\leq e^{-x_{\mathcal{J}}}e^{-\xi }.
\label{etape5}
\end{align}
From \eqref{debut:A1} and \eqref{etape5},
\begin{equation*}
A_{1}(\mathcal{J})\leq \int _{0}^{+\infty } 2t e^{-x_{\mathcal{J}}}e
^{-h_{\mathcal{J}}^{-1}(t)} dt,
\end{equation*}
with
\begin{equation*}
h_{\mathcal{J}}^{-1}(t)=\frac{nt}{4(1+\rho \mbox{Card}(\mathcal{J}))^{2}
M^{2}}.
\end{equation*}
Thus:
%
\begin{align}
A_{1}(\mathcal{J}) \leq
& \int _{0}^{+\infty } 2t e^{-x_{\mathcal{J}}}
\exp \Big (-\frac{nt}{4(1+\rho \mbox{Card}(\mathcal{J}))^{2} M^{2}}
\Big )dt
\nonumber
\\
\leq
& \frac{32(1+\rho \mbox{Card}(\mathcal{J}))^{4} M^{4}}{n^{2}} e
^{-x_{\mathcal{J}}} \leq C \frac{\mbox{Card}^{4}(\mathcal{J}) e^{-x
_{\mathcal{J}}}}{n^{2}}
\nonumber
\\
\leq
& \frac{C \mbox{Card}^{2}(\mathcal{J})}{n^{2}},
\label{etape:fin1}
\end{align}
using $(\sum _{j\in \mathcal{J}}2^{j})^{-1} \leq C/\mbox{Card}^{2}(
\mathcal{J})$.%

From the choice of $x_{\mathcal{J}}$ \eqref{choix:xJ-1}, we have:
\begin{align*}
\sum _{\mathcal{J}\subset \{-1,\dots , J_{n}\}} \mbox{Card}^{2}(
\mathcal{J}) \leq C 2^{J_{n}} J_{n}^{2} \leq C \sqrt{n} \log _{2}
^{2} (n),
\end{align*}
by choice of $J_{n}=\log _{2}\big (\sqrt{n}\big )$. From this and
\eqref{etape:fin1}, we deduce that:
%
\begin{equation}
\sum _{\mathcal{J}\subset \{-1,\dots , J_{n}\}}A_{1}(\mathcal{J})
\leq \frac{C \log _{2}^{2} (n)}{n^{3/2}}.
\end{equation}
\underline{Upper bound of $A_{2}(\mathcal{J})$}

\medskip
For the term $A_{2}(\mathcal{J})$ of \eqref{decompo:A1A2}, we have, for
$(j,k)$ such that $j\not =-1$:
\begin{align*}
|\bar{\gamma }_{n}(\psi _{jk})|\leq
& M2^{j/2} \|\psi \|_{\infty }+ M 2^{-j/2}
\int _{\mathbb{R}}|\psi (u)|du.
\end{align*}
Thus, for a constant $C$ that depends only on the choice of
$\psi _{-10}$ and $\psi _{00}$:
%
\begin{equation}
A_{2}(\mathcal{J}) \leq \Big [C\sum _{j \in \mathcal{J}}\Big (M2^{j/2}
\|\psi \|_{\infty }+ M 2^{-j/2}\int _{\mathbb{R}}|\psi (u)|du\Big )^{2}
\Big ]^{2}\times \mathbb{P}\Big (\Omega ^{c}_{\mathcal{J}}(\rho )\Big ).
\label{etape10}
\end{equation}
Since:
\begin{align*}
& \sum _{i=1}^{n} \mathbb{E}\Big [\Big (\frac{Y_{i} \psi _{jk}\big (G_{
\ell }(X_{\ell }^{i})\big )-\mathbb{E}\big (Y_{1} \psi _{jk}(G_{\ell }(X
_{\ell }^{1}))\big )}{n}\Big )^{2}\Big ]
=
\frac{\mbox{Var}\big (Y_{1} \psi _{jk}(G_{\ell }(X_{\ell }^{1}))\big )}{n}
\leq \frac{M^{2}}{n},
\\
& \Big |\frac{Y_{i} \psi _{jk}\big (G_{\ell }(X_{\ell }^{i})\big )-
\mathbb{E}\big (Y_{1} \psi _{jk}(G_{\ell }(X_{\ell }^{1}))\big )}{n}
\Big |\leq \frac{2 M 2^{j/2} \|\psi \|_{\infty }}{n}\mbox{ a.s.}
\end{align*}
then we have by Bernstein's inequality (e.g.
\cite{massart_concentration}):
\begin{align*}
\mathbb{P}\big (\big |\bar{\gamma }_{n}(\psi _{jk})\big |\geq \rho 2^{-j/2}
\big )\leq 2 \exp \Big (-\frac{n \rho ^{2} 2^{-j}}{2 \big (M^{2}+2 M \|
\psi \|_{\infty }\rho \big )}\Big ).
\end{align*}
As a consequence, recalling that $J_{\mbox{{\scriptsize max}}}=\max
\mathcal{J}$, we have
%
\begin{align}
\sum _{\mathcal{J}\subset \{-1,\dots J_{n}\}} A_{2}(\mathcal{J})\leq
& \sum _{\mathcal{J}\subset \{-1,\dots J_{n}\}} 2^{2J_{
\mbox{{\scriptsize max}}}} \mathbb{P}\Big (\exists (j,k)\in
\mathcal{J}\times \mathbb{Z},\ \big |\bar{\gamma }_{n}(\psi _{jk})
\big |\geq \rho 2^{-j/2} \Big )
\nonumber
\\
\leq
& C \sum _{\mathcal{J}\subset \{-1,\dots J_{n}\}} 2^{3J_{
\mbox{{\scriptsize max}}}} \exp \Big (-\frac{n \rho ^{2} 2^{-J_{
\mbox{{\scriptsize max}}}}}{2 \big (M^{2}+2 M \|\psi \|_{\infty }
\rho \big )}\Big ),
\label{etape:fin1bis}
\end{align}
which is smaller than $C/n^{3/2}$ for sufficiently large $n$, as
$J_{\mbox{{\scriptsize max}}}\leq J_{n}=\log _{2}(\sqrt{n})$.%

\medskip\noindent
\textbf{Step 3: Upper bound of the second term in the r.h.s. of \eqref{etape8}}

\medskip
For the second term in the r.h.s. of \eqref{etape8},
%
\begin{align}
& 2\langle \widehat{h}_{\mathcal{J},\ell }-h_{\mathcal{J},
\ell }, h_{\mathcal{J},\ell } \rangle -\|h_{\ell }-h_{\mathcal{J},
\ell }\|^{2}_{2} -\zeta _{n}-\mbox{pen}_{2}(\mathcal{J})
\nonumber
\\
=
& 2 \sum _{j\in \mathcal{J}} \sum _{k\in \mathbb{Z}} \bar{\gamma }
_{n}\big (\psi _{jk}\big ) \beta _{jk}^{\ell }-\|h_{\ell }-h_{\mathcal{J},
\ell }\|^{2}_{2} -2\bar{\gamma }_{n}\big (h_{\ell }\big )-\mbox{pen}
_{2}(\mathcal{J})
\nonumber
\\
=
& 2 \bar{\gamma }_{n}\Big (\sum _{j\in \mathcal{J}}
\sum _{k\in \mathbb{Z}} \beta _{jk}^{\ell }\psi _{jk}\Big ) -\|h_{\ell }-h
_{\mathcal{J},\ell }\|^{2}_{2} -2\bar{\gamma }_{n}\big (h_{\ell }
\big )-\mbox{pen}_{2}(\mathcal{J})
\nonumber
\\
=
& 2 \bar{\gamma }_{n}\big (h_{\mathcal{J},\ell }-h_{\ell }\big ) -\|h
_{\ell }-h_{\mathcal{J},\ell }\|^{2}_{2} -\mbox{pen}_{2}(\mathcal{J})
\label{etape4:egalite}%
\\
\leq
& \Big (\frac{\bar{\gamma }_{n}\big (h_{\mathcal{J},\ell }-h_{
\ell }\big )}{\|h_{\ell }-h_{\mathcal{J},\ell }\|_{2}}\Big )^{2}-
\mbox{pen}_{2}(\mathcal{J}) = \bar{\gamma }_{n}^{2} \Big (\frac{h_{
\mathcal{J},\ell }-h_{\ell }}{\|h_{\ell }-h_{\mathcal{J},\ell }\|_{2}}
\Big ) -\mbox{pen}_{2}(\mathcal{J})
\label{etape4}%
,
\end{align}
by using the identity $2ab-b^{2}\leq a^{2}$. Setting $
\varphi _{\mathcal{J}}=\frac{h_{\mathcal{J},\ell }-h_{\ell }}{\|h_{
\ell }-h_{\mathcal{J},\ell }\|_{2}}$ and using Bernstein's formula (see
\cite[p.25]{massart_concentration}), we have for all $x>0$:
%
\begin{equation}
\mathbb{P}\Big (\bar{\gamma }_{n} \big (\varphi _{\mathcal{J}}\big )
\geq \sqrt{\frac{2M^{2}}{n}x}+\frac{2M\|\varphi _{\mathcal{J}}\|_{
\infty }}{n}x \Big )\leq e^{-x}.
\end{equation}
Setting $x=x_{\mathcal{J}}+\xi $ with now
%
\begin{equation}
\label{choix:xJ-2}
x_{\mathcal{J}}= \mbox{Card}(\mathcal{J}) \log 2,
\end{equation}
and using that $(a+b)^{2}\leq 2 a^{2}+2b^{2}$, we obtain that
%
\begin{align}
\mathbb{P}\Big (\bar{\gamma }^{2}_{n} \big (\varphi _{\mathcal{J}}\big ) -
\Big [ \frac{4M^{2} }{n} x_{\mathcal{J}}+\frac{16 M^{2} \|
\varphi _{\mathcal{J}}\|_{\infty }^{2}}{n^{2}} x_{\mathcal{J}}^{2}
\Big ]\geq r_{n}(\xi ) \Big )\leq e^{-x_{\mathcal{J}}}e^{-\xi },
\label{etape15}
\end{align}
where and
\begin{equation*}
r_{n}(\xi )=\frac{16 M^{2} \|\varphi _{\mathcal{J}}\|_{\infty }^{2} \xi
^{2}}{n^{2}}+\frac{4M^{2}\xi }{n}.
\end{equation*}
Let us consider the square bracket in \eqref{etape15}. Recall
\eqref{choix:xJ-2}. Because $\mbox{Card}(\mathcal{J})\leq J_{n} =\log
_{2}\big (\sqrt{n})$, $x_{\mathcal{J}}/n$ converges to zero
when $n\rightarrow +\infty $ and it is possible to choose a constant
$K$ and $n_{0}$ sufficiently large such that for all $n\geq n_{0}$,
%
\begin{equation}
\mbox{pen}_{2}(\mathcal{J}) \geq \frac{4M^{2} }{n}
x_{\mathcal{J}}+\frac{16 M^{2} \|\varphi _{\mathcal{J}}\|_{\infty }
^{2}}{n^{2}} x_{\mathcal{J}}^{2},
\end{equation}
where we recall that $\mbox{pen}_{2}(\mathcal{J})$ has been defined in
\eqref{def:penalisations}. Then, this yields
%
\vspace*{-3pt}
\begin{align}
\mathbb{P}\Big (\bar{\gamma }^{2}_{n} \big (\varphi _{\mathcal{J}}\big ) -
\mbox{pen}_{2}(\mathcal{J}) \geq r_{n}(\xi ) \Big )\leq e^{-x_{
\mathcal{J}}}e^{-\xi }.
\label{etape6}
\end{align}
From this, we deduce that
\vspace*{-4pt}
\begin{align}
&\mathbb{E}\Big ( \big (2\langle \widehat{h}_{\mathcal{J},
\ell }-h_{\mathcal{J},\ell }, h_{\mathcal{J},\ell } \rangle -\|h_{
\ell }-h_{\mathcal{J},\ell }\|^{2}_{2} -\zeta _{n}-\mbox{pen}_{2}(\mathcal{J})\big )^{2}_{+}\Big )
\nonumber
\\[-2pt]
\leq
& \mathbb{E}\Big ( \Big [\bar{\gamma }^{2}_{n} \big (
\varphi _{\mathcal{J}}\big )-\mbox{pen}_{2}(\mathcal{J})\Big ]^{2}\Big )
=
\int _{0}^{+\infty } 2t \ \mathbb{P}\big (\bar{\gamma }^{2}_{n} \big (
\varphi _{\mathcal{J}}\big )-\mbox{pen}_{2}(\mathcal{J})>t\big )dt
\nonumber
\\[-2pt]
\leq
& C e^{-x_{\mathcal{J}}} \int _{0}^{+\infty } t\ \exp \Big (-\frac{n}{8
\|\varphi _{\mathcal{J}}\|_{\infty }^{2}} \big (\sqrt{1+\frac{4 t \|
\varphi _{\mathcal{J}}\|_{\infty }^{2}}{M^{2}}}-1\big )\Big )dt
\leq \frac{C e^{-x_{\mathcal{J}}}}{n^{2}}.
\label{etape:fin2}
\end{align}
The last inequality comes from the behaviour of the integrand when
$t$ is close to 0.%

From the choice of $x_{\mathcal{J}}$ \eqref{choix:xJ-2}, we have:
\vspace*{-4pt}
\begin{multline}
\frac{1}{n^{2}}\sum _{\mathcal{J}\subset \{-1,\dots , J_{n}\}} e^{-x
_{\mathcal{J}}} =\frac{1}{n^{2}}
\sum _{\mathcal{J}\subset \{-1,\dots , J_{n}\}} 2^{-\mbox{Card}(
\mathcal{J})}
\\
\leq \frac{C}{n^{2}} \sum _{k=0}^{J_{n}+2} 2^{-k}
{k \choose J_{n}+2 } = \frac{C}{n^{2}}\Big (\frac{3}{2}\Big )^{J_{n}+2}
\leq \frac{C}{n^{3/2}}.
\end{multline}

Gathering the results of Steps 1 to 3, we have by
\eqref{decompo:A1A2} and \eqref{etape8} that the first term in the
r.h.s. of \eqref{etape:but1} is smaller than $C
\log _{2}^{2}(n)/n^{3/2}$. This proves \eqref{oracle:etape1}.%

\medskip\noindent
\textbf{Step 4:}

\smallskip\noindent
Let us now consider the term $\mathbb{E}\Big [\Big (-\widehat{V}_{
\mathcal{J},\ell }+\mbox{pen}(\mathcal{J})+V_{\ell }+ \zeta _{n}
\Big )^{2}_{+}\Big ]$ in \eqref{oracle:etape1}. From \eqref{etape3} and
\eqref{etape4:egalite}:
\vspace*{-4pt}
\begin{align}
&\mathbb{E}\Big [\big (-\widehat{V}_{\mathcal{J},\ell }+\mbox{pen}(\mathcal{J})+V_{\ell }+ \zeta _{n}\big )^{2}_{+}\Big ]
\nonumber
\\
=
& \mathbb{E}\Big ( \big (\|h_{\ell }-h_{\mathcal{J},\ell }\|^{2}_{2}
- \| \widehat{h}_{\mathcal{J},\ell }-h_{\mathcal{J},\ell } \|^{2}_{2}
+2 \bar{\gamma }_{n}\big (h_{\ell }- h_{\mathcal{J},\ell }\big ) +
\mbox{pen}(\mathcal{J})\big )^{2}_{+}\Big )
\nonumber
\\
\leq
& 4 \Big ( \|h_{\ell }-h_{\mathcal{J},\ell }\|^{4}_{2} + 4
\mathbb{E}\Big (\bar{\gamma }^{2}_{n}\big (h_{\ell }- h_{\mathcal{J},
\ell }\big ) \Big )
+ \mathbb{E}\Big (\Big [ \| \widehat{h}_{\mathcal{J},\ell }-h_{
\mathcal{J},\ell } \|^{2}_{2}-\mbox{pen}_{1}(\mathcal{J})\Big ]^{2}
_{+}\Big )+ \mbox{pen}_{2}^{2}(\mathcal{J})\Big ).
\label{etape11}
\end{align}

For the second term in the r.h.s. of \eqref{etape11}, we have:
\vspace*{-4pt}
\begin{align}
\mathbb{E}\Big (\bar{\gamma }^{2}_{n}\big (h_{\ell }- h_{\mathcal{J},
\ell }\big ) \Big )=
& \mbox{Var}\Big ( \bar{\gamma }_{n}\big (h_{\ell }-
h_{\mathcal{J},\ell }\big )\Big )
\nonumber
\\
\leq
& \frac{1}{n} \mathbb{E}\Big (Y_{1}^{2} \big (h_{\ell }(G_{\ell }(X
^{1}_{\ell })) - h_{\mathcal{J},\ell }(G_{\ell }(X^{1}_{\ell }))
\big )^{2}\Big )
\nonumber
\\
\leq
& \frac{M^{2} \|h_{\ell }-h_{\mathcal{J},\ell }\|_{2}^{2}}{n}
\leq C \big (\frac{1}{n^{2}}+\|h_{\ell }-h_{\mathcal{J},\ell }\|_{2}
^{4}\big )
\label{etape12}
\end{align}
by using that $2ab\leq a^{2}+b^{2}$ for the last inequality.%

The third term in the r.h.s. of \eqref{etape11} has been treated in
\eqref{decompo:A1A2} previously. We established an upper bound in
$ \mbox{Card}^{2}(\mathcal{J}) /n^{2} $ (see \eqref{etape:fin1}). For
the fourth term, $\mbox{pen}_{2}^{2}(\mathcal{J})=K^{2} \log ^{2}(2)
\mbox{Card}^{2}(\mathcal{J})/n^{2}$ from \eqref{def:penalisations}.
Gathering these results, we obtain \eqref{oracle:etape2} and then
\eqref{oracle}.

\subsection{Proof of Corollary \ref{corol:vitesseconvergence}}

Plugging \eqref{oracle:etape2} in \eqref{oracle:etape1}, and using that
%
\begin{equation}
\mathbb{E}\big (\zeta ^{2}_{n}\big )= \frac{2}{n} \mbox{Var}\Big (Y_{1} h
_{\ell }\big (G_{\ell }(X^{1}_{\ell })\big )\Big )\leq \frac{2M^{2} \|h
_{\ell }\|_{2}^{2}}{n},
\end{equation}
we obtain:
%
\begin{equation}
\label{etape13}
\mathbb{E}\Big [\big (\widehat{V}_{\ell }-V_{\ell }\big )^{2}\Big ]\leq C
\Big [\inf _{\mathcal{J}\subset \{-1,\dots , J_{n}\}} \Big ( \|h_{\ell }-h
_{\mathcal{J},\ell }\|_{2}^{4} + \frac{\mbox{Card}^{2}(\mathcal{J})}{n
^{2}} \Big ) + \frac{1+\|h_{\ell }\|_{2}^{2}}{n}\Big ].
\end{equation}
If $h_{\ell }\in \mathcal{B}(\alpha ,2,\infty )$, then from Proposition
\ref{prop:besov}, we have for $\mathcal{J}=\{-1,\dots , J_{
\mbox{{\scriptsize max}}}\}$ that $\|h_{\ell }-h_{\mathcal{J},\ell }
\|_{2}^{4} \leq 2^{-4 \alpha \ J_{\mbox{{\scriptsize max}}}}$. Also, we
have seen that $\mbox{Card}^{2}(\mathcal{J})\leq C 2^{J_{
\mbox{{\scriptsize max}}}}$. Thus, for subsets $\mathcal{J}$ of the form
considered, the infimum is attained when choosing $J_{\mbox{{\scriptsize
max}}}= \frac{2}{4\alpha +1} \log _{2}(n)$. In this case, the infimum in
\eqref{etape13} is upper bounded by $n^{-8\alpha /(4\alpha +1)}$.

For $h_{\ell }$ in a ball of radius $R$, $\|h_{\ell }\|_{2}^{2}\leq R
^{2}$, and we can find an upper bound that does not depend on $h$.
Because the last term in \eqref{etape13} is in $1/n$, the elbow effect
is obtained by comparing the order of the first term in the r.h.s. ($n
^{-8\alpha /(4\alpha +1)}$) with $1/n$ when $\alpha $ varies.\qed


\maketitle


\appendix{}

\section{Sobol indices}

The Sobol indices are based on the following decomposition for $f$ (see
Sobol \cite{sobol}). We recall the formulas here, with the notation
$X_{p+1}$ for the random variable $\varepsilon $:
%
\begin{align}
Y =
& f(X_{1},\dots , X_{p},\varepsilon )
\nonumber
\\
=
& f_{0} + \sum _{\ell =1}^{p+1} f_{\ell }(X_{\ell }) +
\sum _{1 \leq \ell _{1} < \ell _{2} \leq p+1} f_{\ell _{1} \ell _{2}}(X
_{\ell _{1}}, X_{\ell _{2}}) + \dots
+ f_{1, \dots ,p+1}(X_{1},\dots ,X_{p},\varepsilon )
\end{align}
where
\begin{align*}
& f_{0} = \mathbb{E}[Y],
\qquad
f_{\ell }(X_{\ell }) = \mathbb{E}[Y|X_{\ell }] - \mathbb{E}[Y],
\nonumber
\\
& f_{\ell _{1} \ell _{2}}(X_{\ell _{1}}, X_{\ell _{2}}) = \mathbb{E}[Y|X
_{\ell _{1}},X_{\ell _{2}}]-\mathbb{E}[Y|X_{\ell _{1}}]-\mathbb{E}[Y|X
_{\ell _{2}}]-\mathbb{E}[Y],\quad \dots
\nonumber
\end{align*}
Then, the variance of $Y$ can be written as:
%
\begin{equation}
\mbox{Var}(Y)=\sum _{\ell =1}^{p+1}V_{\ell }+
\sum _{1 \leq \ell _{1} < \ell _{2} \leq p+1} V_{\ell _{1} \ell _{2}} +
\dots + V_{1 \dots p+1}
\end{equation}
where
%
\begin{align}
& V_{\ell }= \mbox{Var}(E[Y|X_{\ell }]),
\qquad
V_{\ell _{1} \ell _{2}} = \mbox{Var}(E[Y|X_{\ell _{1}}, X_{\ell _{2}}]) -
V_{\ell _{1}} - V_{\ell _{2}}, \dots
\nonumber
\\
& V_{1\dots p+1} = \mbox{Var}(Y) - \sum _{\ell =1}^{p+1} V_{\ell }-
\sum _{1 \leq \ell _{1} < \ell _{2} \leq p+1} V_{\ell _{1} \ell _{2}} -
\dots
- \sum _{1 \leq \ell _{1} < \dots < \ell _{p} \leq p+1} V_{\ell _{1}
\dots \ell _{p}}
\end{align}
The first order indices are then defined as:
%
\begin{equation}
S_{\ell }=V_{\ell }/\mbox{Var}(Y)=\mbox{Var}(E[Y|X_{\ell }])/
\mbox{Var}(Y)
\end{equation}
$S_{\ell }$ corresponds to the part of the variance that can be
explained by the variance of $Y$ due to the variable $X_{\ell }$ alone.
In the same manner, we define the second order indices, third order
indices, etc. by dividing the variance terms by $\mbox{Var}(Y)$.




\end{document}